\documentclass[reqno,a4paper]{amsart}

\usepackage{amsthm}
\usepackage{amsmath}
\usepackage[all]{xy} \SelectTips{eu}{}
\usepackage{latexsym,amsfonts,amssymb}
\usepackage{hyperref}
\usepackage{mathptmx}
\usepackage{nicefrac}
\usepackage{array}
\usepackage{xcolor}
\usepackage{mathrsfs}
\usepackage{rotating}

\renewcommand{\theequation}{$\smash{\sharp}\mspace{0.5mu}$\arabic{equation}}

\newcommand{\numberseries}{\mdseries}   

\newlength{\thmtopspace}                
\newlength{\thmbotspace}                
\newlength{\thmheadspace}               
\newlength{\thmindent}                  

\newtheoremstyle{bfupright head,slanted body}
                {\thmtopspace}{\thmbotspace}
                {\slshape}{\thmindent}{\bfseries}{.}{\thmheadspace}
                {{\numberseries \thmnumber{{\bf #2} }}\thmnote{#3}}

\newtheoremstyle{bfupright head,upright body}
                {\thmtopspace}{\thmbotspace}
                {\upshape}{\thmindent}{\bfseries}{.}{\thmheadspace}
                {{\numberseries \thmnumber{{\bf #2} }}\thmnote{#3}}

\newtheoremstyle{bfit head,upright body}
                {\thmtopspace}{\thmbotspace}
                {\upshape}{\thmindent}{\upshape}{.}{\thmheadspace}
                {{\numberseries\thmnumber{{\bf #2} }}
                {\bfseries\itshape\thmnote{\negthickspace#3}}}

\newtheoremstyle{it head,upright body}
                {\thmtopspace}{\thmbotspace}
                {\upshape}{\thmindent}{\upshape}{.}{\thmheadspace}
                {{\numberseries\thmnumber{{\bf #2} }}
                {\itshape\thmnote{\negthickspace#3}}}

\newtheoremstyle{fixed bf head,slanted body}
                {\thmtopspace}{\thmbotspace}{\slshape}
                {\thmindent}{\bfseries}{.}{\thmheadspace}
                {{\numberseries \thmnumber{{\bf #2} }}\thmname{#1}\thmnote{ (#3)}}

\newtheoremstyle{fixed bf head,upright body}
                {\thmtopspace}{\thmbotspace}{\upshape}
                {\thmindent}{\bfseries}{.}{\thmheadspace}
                {{\numberseries \thmnumber{{\bf #2} }}\thmname{#1}\thmnote{ (#3)}}

\newtheoremstyle{independent paragraph}
                {\thmtopspace}{\thmbotspace}
                {\upshape}{\thmindent}{\upshape}{}{0pt}
                {\thmnote{#3 }}

\newtheoremstyle{subparagraph}
                {\thmbotspace}{\thmbotspace}
                {\upshape}{\thmindent}{\upshape}{}{0pt}
                {\thmnote{#3 }}

\newtheoremstyle{notes}
                {\thmtopspace}{\thmbotspace}
                {\ttfamily}{\thmindent}{\ttfamily\small }{}{0pt}
                {\thmnote{#3 }}
                
\newtheoremstyle{numbered paragraph}
                {\thmtopspace}{\thmbotspace}{\upshape}
                {\thmindent}{\upshape}{}{\thmheadspace}
                {{\numberseries \thmnumber{\bf #2.}}}

\theoremstyle{bfupright head,slanted body}
\newtheorem{res}{}[section]             \newtheorem*{res*}{}

\theoremstyle{bfit head,upright body}
                 \newtheorem*{com*}{}

\theoremstyle{bfupright head,upright body}
\newtheorem{bfhpg}[res]{}               \newtheorem*{bfhpg*}{}

\theoremstyle{it head,upright body}
               \newtheorem*{ithpg*}{}

\theoremstyle{fixed bf head,slanted body}
\newtheorem{thm}[res]{Theorem}          \newtheorem*{thm*}{Theorem}
\newtheorem{prp}[res]{Proposition}      \newtheorem*{prp*}{Proposition}
\newtheorem{cor}[res]{Corollary}        \newtheorem*{cor*}{Corollary}
\newtheorem{lem}[res]{Lemma}            \newtheorem*{lem*}{Lemma}

\theoremstyle{fixed bf head,upright body}
\newtheorem{dfn}[res]{Definition}       \newtheorem*{dfn*}{Definition}
      \newtheorem*{obs*}{Observation}
\newtheorem{rmk}[res]{Remark}           \newtheorem*{rmk*}{Remark}
\newtheorem{exa}[res]{Example}          \newtheorem*{exa*}{Example}
         \newtheorem*{exe*}{Exercise}
\newtheorem{stp}[res]{Setup}            \newtheorem{stp*}{Setup}
         \newtheorem{ntn*}{Notation}
\newtheorem{con}[res]{Construction}     \newtheorem{con*}{Construction}

\theoremstyle{numbered paragraph}
\newtheorem{ipg}[res]{}

\theoremstyle{subparagraph}

\theoremstyle{notes}

\newlength{\thmlistleft}        
\newlength{\thmlistright}       
\newlength{\thmlistpartopsep}   
\newlength{\thmlisttopsep}      
\newlength{\thmlistparsep}      
\newlength{\thmlistitemsep}     

\newcounter{eqc} 
  {\end{list}}%

\newcounter{prt}
\newenvironment{prt}{\begin{list}{\upshape (\alph{prt})}%
    {\usecounter{prt}%
      \setlength{\leftmargin}{\thmlistleft}%
      \setlength{\labelwidth}{\thmlistleft}%
      \setlength{\rightmargin}{\thmlistright}%
      \setlength{\partopsep}{\thmlistpartopsep}%
      \setlength{\topsep}{\thmlisttopsep}%
      \setlength{\parsep}{\thmlistparsep}%
      \setlength{\itemsep}{\thmlistitemsep}}}%
  {\end{list}}%

\newcommand{\prtlbl}[1]{{\upshape(#1)}}

\newcounter{rqm}
\newenvironment{rqm}{\begin{list}{\upshape (\arabic{rqm})}%
    {\usecounter{rqm}%
      \setlength{\leftmargin}{\thmlistleft}%
      \setlength{\labelwidth}{\thmlistleft}%
      \setlength{\rightmargin}{\thmlistright}%
      \setlength{\partopsep}{\thmlistpartopsep}%
      \setlength{\topsep}{\thmlisttopsep}%
      \setlength{\parsep}{\thmlistparsep}%
      \setlength{\itemsep}{\thmlistitemsep}}}%
  {\end{list}}%

\newcommand{\rqmlbl}[1]{{\upshape(#1)}}

\newcounter{rqmm}
  {\end{list}}%

  {\end{list}}%

\newenvironment{itemlist}{\nopagebreak \begin{list}{{\small $\bullet$}}%
    {\setlength{\leftmargin}{\thmlistleft}%
      \setlength{\labelwidth}{\thmlistleft}%
      \setlength{\rightmargin}{\thmlistright}%
      \setlength{\partopsep}{\thmlistpartopsep}%
      \setlength{\topsep}{\thmlisttopsep}%
      \setlength{\parsep}{\thmlistparsep}%
      \setlength{\itemsep}{\thmlistitemsep}}}%
  {\end{list}}%

  \newcommand{\proofoftag}[2][:]{(#2)#1}

\newcommand{\pgref}[1]{\ref{#1}}

\renewcommand{\eqref}[1]{(\pgref{eq:#1})}

\newcommand{\corref}[2][Corollary ]{#1\pgref{cor:#2}}
\newcommand{\dfnref}[2][Definition~]{#1\pgref{dfn:#2}}
\newcommand{\exaref}[2][Example ]{#1\pgref{exa:#2}}
\newcommand{\lemref}[2][Lemma ]{#1\pgref{lem:#2}}

\newcommand{\prpref}[2][Proposition ]{#1\pgref{prp:#2}}

\newcommand{\rmkref}[2][Remark ]{#1\pgref{rmk:#2}}
\newcommand{\thmref}[2][Theorem ]{#1\pgref{thm:#2}}
\newcommand{\stpref}[2][Setup~]{#1\pgref{stp:#2}}
\newcommand{\secref}[2][Section ]{#1\pgref{sec:#2}}
\newcommand{\appref}[2][Appendix~]{#1\ref{app:#2}}

\newcommand{\conref}[2][Construction~]{#1\ref{con:#2}}

\makeatletter
\def\@nobreak@#1{\mathchoice%
  {\nobreakdef@\displaystyle\f@size{#1}}%
  {\nobreakdef@\nobreakstyle\tf@size{\firstchoice@false #1}}%
  {\nobreakdef@\nobreakstyle\sf@size{\firstchoice@false #1}}%
  {\nobreakdef@\nobreakstyle\ssf@size{\firstchoice@false #1}}%
  \check@mathfonts}%
\def\nobreakdef@#1#2#3{\hbox{{%
                    \everymath{#1}%
                    \let\f@size#2\selectfont%
                    #3}}}%
\makeatother

\DeclareFontFamily{T1}{cmex}{}
\DeclareFontShape{T1}{cmex}{m}{n}{<-> s * [0.89] cmex10}{}
\DeclareSymbolFont{cmlargesymbols}{T1}{cmex}{m}{n}
\DeclareMathSymbol{\mycoprod}{\mathop}{cmlargesymbols}{"60} 
\DeclareMathSymbol{\myprod}{\mathop}{cmlargesymbols}{"51} \let\prod\myprod

\DeclareSymbolFont{usualmathcal}{OMS}{cmsy}{m}{n}
\DeclareSymbolFontAlphabet{\mathcal}{usualmathcal}

\DeclareSymbolFont{letters}{OML}{txmi}{m}{it}

\DeclareMathSymbol{\alpha}{\mathord}{letters}{"0B}
\DeclareMathSymbol{\beta}{\mathord}{letters}{"0C}
\DeclareMathSymbol{\gamma}{\mathord}{letters}{"0D}
\DeclareMathSymbol{\sigma}{\mathord}{letters}{"0E}
\DeclareMathSymbol{\epsilon}{\mathord}{letters}{"0F}
\DeclareMathSymbol{\zeta}{\mathord}{letters}{"10}
\DeclareMathSymbol{\eta}{\mathord}{letters}{"11}
\DeclareMathSymbol{\theta}{\mathord}{letters}{"12}
\DeclareMathSymbol{\iota}{\mathord}{letters}{"13}
\DeclareMathSymbol{\kappa}{\mathord}{letters}{"14}
\DeclareMathSymbol{\lambda}{\mathord}{letters}{"15}
\DeclareMathSymbol{\mu}{\mathord}{letters}{"16}
\DeclareMathSymbol{\nu}{\mathord}{letters}{"17}
\DeclareMathSymbol{\xi}{\mathord}{letters}{"18}
\DeclareMathSymbol{\pi}{\mathord}{letters}{"19}
\DeclareMathSymbol{\rho}{\mathord}{letters}{"1A}
\DeclareMathSymbol{\sigma}{\mathord}{letters}{"1B}
\DeclareMathSymbol{\tau}{\mathord}{letters}{"1C}
\DeclareMathSymbol{\upsilon}{\mathord}{letters}{"1D}
\DeclareMathSymbol{\phi}{\mathord}{letters}{"1E}
\DeclareMathSymbol{\chi}{\mathord}{letters}{"1F}
\DeclareMathSymbol{\psi}{\mathord}{letters}{"20}
\DeclareMathSymbol{\omega}{\mathord}{letters}{"21}
\DeclareMathSymbol{\varepsilon}{\mathord}{letters}{"22}
\DeclareMathSymbol{\vartheta}{\mathord}{letters}{"23}
\DeclareMathSymbol{\varpi}{\mathord}{letters}{"24}
\DeclareMathSymbol{\varrho}{\mathord}{letters}{"25}
\DeclareMathSymbol{\varsigma}{\mathord}{letters}{"26}
\DeclareMathSymbol{\varphi}{\mathord}{letters}{"27}

\DeclareMathSymbol{\Gamma}{\mathord}{letters}{"00}
\DeclareMathSymbol{\Delta}{\mathord}{letters}{"01}
\DeclareMathSymbol{\Theta}{\mathord}{letters}{"02}
\DeclareMathSymbol{\Lambda}{\mathord}{letters}{"03}
\DeclareMathSymbol{\Xi}{\mathord}{letters}{"04}
\DeclareMathSymbol{\Pi}{\mathord}{letters}{"05}
\DeclareMathSymbol{\Sigma}{\mathord}{letters}{"06}
\DeclareMathSymbol{\Upsilon}{\mathord}{letters}{"07}
\DeclareMathSymbol{\Phi}{\mathord}{letters}{"08}
\DeclareMathSymbol{\Psi}{\mathord}{letters}{"09}
\DeclareMathSymbol{\Omega}{\mathord}{letters}{"0A}

\DeclareMathSymbol{\upGamma}{\mathalpha}{operators}{"00}
\DeclareMathSymbol{\upDelta}{\mathalpha}{operators}{"01}
\DeclareMathSymbol{\upTheta}{\mathalpha}{operators}{"02}
\DeclareMathSymbol{\upLambda}{\mathalpha}{operators}{"03}
\DeclareMathSymbol{\upXi}{\mathalpha}{operators}{"04}
\DeclareMathSymbol{\upPi}{\mathalpha}{operators}{"05}
\DeclareMathSymbol{\upSigma}{\mathalpha}{operators}{"06}
\DeclareMathSymbol{\upUpsilon}{\mathalpha}{operators}{"07}
\DeclareMathSymbol{\upPhi}{\mathalpha}{operators}{"08}
\DeclareMathSymbol{\upPsi}{\mathalpha}{operators}{"09}
\DeclareMathSymbol{\upOmega}{\mathalpha}{operators}{"0A}

\DeclareMathAlphabet\PazoBB{U}{fplmbb}{m}{n}%

\newcommand{\cofr}{\mathscr{Q}}
\newcommand{\fibr}{\mathscr{R}}

\newcommand{\cpx}[1]{#1_{\scriptscriptstyle{\bullet}}}

\newcommand{\uHom}[3]{\underline{\operatorname{Hom}}_{#1}(#2,#3)}
\newcommand{\Hom}[3]{\operatorname{Hom}_{#1}(#2,#3)}
\newcommand{\Ext}[4]{\operatorname{Ext}_{#1}^{#2}(#3,#4)}
\newcommand{\Tor}[4]{\operatorname{Tor}^{#1}_{#2}(#3,#4)}
\renewcommand{\Im}[1]{\operatorname{Im}\mspace{1mu}#1}
\newcommand{\Ker}[1]{\operatorname{Ker}\mspace{1mu}#1}
\newcommand{\Coker}[1]{\operatorname{Cok}\mspace{1mu}#1}

\newcommand{\lMod}[1]{{}_{#1}\mspace{-1mu}\operatorname{Mod}}
\newcommand{\rMod}[1]{\operatorname{Mod}_{#1}}
\newcommand{\lGPrj}[1]{{}_{#1}\mspace{-1mu}\operatorname{GPrj}}
\newcommand{\lPrj}[1]{{}_{#1}\mspace{-1mu}\operatorname{Prj}}
\newcommand{\lInj}[1]{{}_{#1}\mspace{-1mu}\operatorname{Inj}}
\newcommand{\Serre}{\mathbb{S}}

\newcommand{\QSD}[2]{\mathcal{D}_{#1}(#2)}
\newcommand{\QSDperf}[2]{\mathcal{D}^\mathrm{perf}_{#1}(#2)}

\newcommand{\stalkco}[1]{S\mspace{-2mu}\langle{#1}\rangle}
\newcommand{\stalkcn}[1]{S\mspace{-2mu}\{#1\}}

\newcommand{\alg}{A}

\newcommand{\vtx}{\text{\small $\bullet$}}

\newcommand{\Cq}[1]{C_{\mspace{-1mu}\smash{#1}}}
\newcommand{\Sq}[1]{S_{\mspace{-4mu}\smash{#1}}}
\newcommand{\Kq}[1]{K_{\smash{#1}}}
\newcommand{\Fq}[1]{F_{\mspace{-5mu}\smash{#1}}}
\newcommand{\Eq}[1]{E_{\smash{#1}}}
\newcommand{\Gq}[1]{G_{\mspace{-1mu}\smash{#1}}}

\newcommand{\ORt}[3][Q]{#2 \otimes_{#1} #3}

\newcommand{\hH}[2][1]{\mathbb{H}_{#1}^{\textnormal{\tiny[}#2\textnormal{\tiny]}}}
\newcommand{\cH}[2][1]{\mathbb{H}^{#1}_{\textnormal{\tiny[}#2\textnormal{\tiny]}}}

\newcommand{\lE}[1]{{}_{#1}\mathscr{E}}
\newcommand{\finsupp}[1]{{}_{#1}\mathscr{\mathcal{K}}}
\newcommand{\supp}[1]{\operatorname{supp}\mspace{1mu}#1}
\newcommand{\Filt}[1]{\operatorname{Filt}(#1)}
\newcommand{\coFilt}[1]{\operatorname{coFilt}(#1)}

\begin{document}

\title{The $Q$-shaped derived category of a ring --- compact and perfect objects}

\author{Henrik Holm \ }

\address{Department of Mathematical Sciences, University of Copenhagen, Universitetsparken~5, 2100 Copenhagen {\O}, Denmark} 
\email{holm@math.ku.dk}

\urladdr{http://www.math.ku.dk/\~{}holm/}

\author{\ Peter J{\o}rgensen}

\address{Department of Mathematics, Aarhus University, Ny Munkegade 118, Building 1530, 8000 Aarhus C, Denmark}
\email{peter.jorgensen@math.au.dk}

\urladdr{https://sites.google.com/view/peterjorgensen}




\vspace*{-2ex}

\begin{abstract}
  In a previous work we constructed the $Q$-shaped derived category of any ring $\alg$ for any suitably nice category $Q$. The $Q$-shaped derived category of $\alg$, which is denoted by \smash{$\QSD{Q}{\alg}$}, is a generalization of the ordinary derived category. In this paper we prove that the $Q$-shaped derived category of $\alg$ is a compactly generated triangulated category. We also define perfect objects in \smash{$\QSD{Q}{\alg}$} and prove that these constitute a triangulated~subcategory, \smash{$\QSDperf{Q}{\alg}$}, of the category \smash{$\QSD{Q}{\alg}^\mathrm{c}$} of compact objects in the $Q$-shaped derived category. The subcategories \smash{$\QSDperf{Q}{\alg}$} and \smash{$\QSD{Q}{\alg}^\mathrm{c}$} coincide if and only if the former is thick.
\end{abstract}

\maketitle

\section{Introduction}
\label{sec:Introduction}

A chain complex over a ring $\alg$ can be viewed as an $\lMod{\alg}$-valued representation of the quiver
\begin{equation}
\label{eq:Gamma}
  \upGamma \,= \ \xymatrix{
    \cdots \ar[r]^-{\partial} &
    \vtx
    \ar[r]^-{\partial} &
    \vtx
    \ar[r]^-{\partial} &
    \vtx
    \ar[r]^-{\partial} &
    \cdots
  }
\end{equation}
with the relations $\partial^2 = 0$.  Expressed more formally, the category of chain complexes $\operatorname{Ch}\mspace{1mu}\alg$ can be identified with $\lMod{Q,\,\alg}$, the category of additive functors $Q \to \lMod{\alg}$, where $Q$ is the path category of $\upGamma$ with these relations.

A key property of $\upGamma$ is that it is a {\em self-injective} quiver with relations; equivalently, $Q$ has a Serre functor (see \ref{Bbbk}).  We showed in \cite{HJ2021} that many properties of chain complexes can be generalized to $\lMod{Q,\,\alg}$ when $Q$ is another suitable category with a Serre functor.  This permits $Q$ to be the path category of many different quivers with relations.  Among them are cyclic quivers, like:
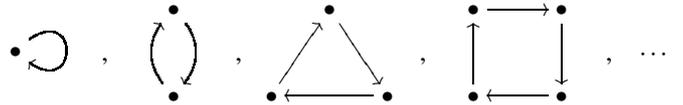
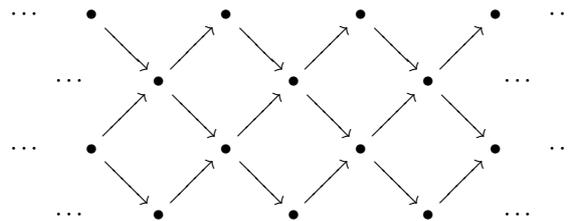
\begin{equation*}
    \begin{gathered}
    \xymatrix{
      \vtx \ar@(ur,dr)[]
    }
    \end{gathered}
    \hspace*{6ex} , \hspace*{4ex}
    \begin{gathered}
    \xymatrix@R=5ex{
      \vtx \ar@/^0.7pc/[d] 
      \\
      \vtx \ar@/^0.7pc/[u]
    }
    \end{gathered}    
    \hspace*{4ex} , \hspace*{1ex}
    \begin{gathered}
    \xymatrix@C=2.5ex@R=5ex{
      {} & 
      \vtx \ar[dr] 
      & {} 
      \\
      \vtx 
      \ar[ur]
      & {} & 
      \vtx
      \ar[ll]
    }
    \end{gathered}
    \hspace*{1.5ex} , \hspace*{2.5ex}
    \begin{gathered}
    \xymatrix@C=5ex@R=5ex{
      \vtx \ar[r] & 
      \vtx \ar[d] 
      \\
      \vtx \ar[u] & 
      \vtx \ar[l]
    }
    \end{gathered}
    \hspace*{2.5ex} , \hspace*{2ex} \cdots \hspace*{1ex},
\end{equation*}
repetitive quivers, like:
  \begin{equation*}
    \begin{gathered}
    \xymatrix@!=0.1pc{
      \cdots &
      \vtx \ar[dr] & {} & 
      \vtx \ar[dr] & {} & 
      \vtx \ar[dr] & {} & 
      \vtx
      &
      \cdots
      \\
      {} & \mspace{-30mu}\cdots & 
      \vtx 
      \ar[ur]
      \ar[dr] & 
      {} & 
      \vtx 
      \ar[ur] 
      \ar[dr] & 
      {} & 
      \vtx 
      \ar[ur] 
      \ar[dr] & 
      \mspace{30mu}\cdots & {}
      \\
      \cdots & \vtx 
      \ar[ur]
      \ar[dr] & 
      {} & 
      \vtx 
      \ar[ur]
      \ar[dr] & 
      {} & 
      \vtx 
      \ar[ur]
      \ar[dr] & 
      {} &
      \vtx & \cdots
      \\
      {} & \mspace{-30mu}\cdots & 
      \vtx 
      \ar[ur]
      & 
      {} & 
      \vtx 
      \ar[ur]
      & 
      {} & 
      \vtx 
      \ar[ur]
      & 
      \mspace{30mu}\cdots & {}
    }
    \end{gathered}    
  \end{equation*} 
and double quivers, like:
\begin{equation*}
  \xymatrix@C=2.5pc{
    \vtx \ar@<0.6ex>[r] &
    \vtx \ar@<0.6ex>[l] \ar@<0.6ex>[r] &
    \ \cdots \ \ar@<0.6ex>[r] \ar@<0.6ex>[l] &
    \vtx \ar@<0.6ex>[l] \ar@<0.6ex>[r] &
    \vtx \ar@<0.6ex>[l]
  }\,.
\end{equation*}

\noindent
To wit, the following was proved in \cite{HJ2021} when $Q$ has a Serre functor.  
\begin{itemlist}

  \item There is a notion of ``quasi-isomorphism'' in $\lMod{Q,\,\alg}$, which generalizes the usual notion of quasi-isomorphism in  $\operatorname{Ch}\mspace{1mu}\alg$. Inverting the quasi-isomorphisms in $\lMod{Q,\,\alg}$ gives a category $\QSD{Q}{\alg}$, which we call the \emph{$Q$-shaped derived category of $\alg$}.
  
  \item  $\QSD{Q}{\alg}$ is a triangulated category.  If $Q$ is the path category of \eqref{Gamma} with the relations $\partial^2 = 0$, then $\QSD{Q}{\alg}$ can be identified with the classic derived category $\QSD{}{\alg}$.

  \item  $\lMod{Q,\,\alg}$ contains three classes of objects, $\lE{}$, ${}^\perp\lE{}$, and $\lE{}^\perp$, which generalize the exact, semi-projective, and semi-injective complexes (see \exaref{Ch-semi}).
  
  \item  These classes provide $\lMod{Q,\,\alg}$ with two model category structures, which generalize the standard projective and injective model category structures on $\operatorname{Ch}\mspace{1mu}\alg$ (see \ref{summary-of-HJ2021}).

  \item  Either model category structure on $\lMod{Q,\,\alg}$ has the associated homotopy category
$\operatorname{Ho}(\lMod{Q,\,\alg}) \,=\, \QSD{Q}{\alg}$.

\end{itemlist}

However, there are key properties of the (classic) derived category $\QSD{}{\alg}$ which do not generalize to $\QSD{Q}{\alg}$, and the purpose of this paper is to compare and contrast these categories by investigating several key classes of objects.

\begin{bfhpg}[Strictly perfect objects]
Recall that a {\em strictly perfect} complex $P$ is a bounded complex of finitely generated projective modules.  It is well-known that a strictly perfect complex is {\em compact} in $\QSD{}{\alg}$, that is, $\Hom{\QSD{}{\alg}}{P}{-}$ preserves set-indexed coproducts.  However, as shown in \thmref{counterexample}, the corresponding property does {\em not} always hold for $\QSD{Q}{\alg}$.
\end{bfhpg}

\begin{res*}[Theorem~A]
Let $Q$ be the path category of the Jordan quiver
\begin{equation*}
    \xymatrix{
      \vtx \ar^{\partial}@(ur,dr)[]
    }
\end{equation*}
with the relation $\partial^2 = 0$.  There exists a commutative noetherian ring $A$ and an object 
\begin{equation*}
    M \,= \;
    \xymatrix{
      A \ar^{x}@(ur,dr)[]
    }
\end{equation*}
in $\QSD{Q}{\alg}$ which is \emph{\emph{not}} compact in $\QSD{Q}{\alg}$, despite having the finitely generated projective module $A$ at the (unique) vertex.
\end{res*}

\noindent
Note that for this quiver, $\lMod{Q,\,\alg}$ is the category of {\em differential $A$-modules}, i.e.~$A$-modules equipped with an endomorphism squaring to zero.  Differential modules have been studied intensively in the literature, see for instance \cite{ABI-07}, but Theorem~A is new.

The object $M$ in Theorem~A is an obvious generalization of a strictly perfect complex.  Indeed, we will say that a functor $K \colon Q \to \lMod{\alg}$ 
in $\lMod{Q,\,\alg}$ is \emph{strictly perfect} if it has finite support, that is, if the set $\supp{K} = \{q \in Q \ | \ K(q) \neq 0 \}$ is finite, and the $\alg$-module $K(q)$ is finitely generated and projective for each $q \in Q$ (see \dfnref{perfect}).  As shown by Theorem~A, strictly perfect objects need not be compact in $\QSD{Q}{\alg}$, but we do have the following (positive) result, which is proved in \secref{Perfect}.  For the notion of \emph{cycles} in the category $Q$, see \dfnref{cycle} under the assumptions in \stpref{nilpotent}.  

\begin{res*}[Theorem~B]
  Assume \stpref{nilpotent} and one of the next conditions:
  \begin{rqm}
  \item The category $Q$ has no cycles, or
  \item The ring $\alg$ has finite left global dimension.
  \end{rqm}
  The following assertions hold:
  \begin{prt}
  \item Every strictly perfect object is compact in $\QSD{Q}{\alg}$.
  \item Every strictly perfect object is cofibrant in the projective model structure on $\lMod{Q,\,\alg}$; that is, every strictly perfect object is in \smash{${}^\perp\lE{}$} (see \dfnref{perfect}). 

  \item An object in $\QSD{Q}{\alg}$ is perfect if and only if it is isomorphic to a strictly perfect~\mbox{object.}
  \end{prt}
\end{res*}

\noindent
Notice that in the special case of chain complexes, the category $Q$ has no cycles and hence condition \prtlbl{1} in Theorem~B holds. Also note that by condition \rqmlbl{2} in Theorem~B, the ring in Theorem~A must necessarily have infinite global dimension (and, indeed, it has). We now turn to perfect objects in $\QSD{Q}{\alg}$, which are already mentioned in part \prtlbl{c} of Theorem~B.

\begin{bfhpg}[Perfect objects]
Since strictly perfect objects in the category $\lMod{Q,\,\alg}$ need not be compact in $\QSD{Q}{\alg}$, we propose in \dfnref[]{perfect} the following definition:
\begin{equation*}
  \qquad
  \text{$X \in \QSD{Q}{\alg}$ is \emph{perfect}} \ \ \ \iff \ \ \
  \left\{\!\!
  \begin{array}{l}
  \text{There is an isomorphism $X \cong K$ in $\QSD{Q}{\alg}$}
  \\
  \text{for some strictly perfect object $K \in {}^\perp\lE{}$\;.}
  \end{array}
  \right.
\end{equation*}
Observe that if condition (1) or (2) in Theorem~B holds, then by part \prtlbl{c} of that theorem, the perfect objects are just those which are isomorphic in $\QSD{Q}{\alg}$ to a strictly perfect object.

With this definition, we show in \secref{Perfect} the following result where \smash{$\QSDperf{Q}{\alg}$} is the class of perfect objects 
and $\QSD{Q}{\alg}^\mathrm{c}$ is the class of compact objects in $\QSD{Q}{\alg}$:
\end{bfhpg}

\enlargethispage{3ex}

\begin{res*}[Theorem~C]
Assume \stpref{nilpotent}. There is an inclusion of triangulated subcategories of the $Q$-shaped derived category of $\alg$,
\begin{equation*}
  \QSDperf{Q}{\alg} \,\subseteq\, \QSD{Q}{\alg}^\mathrm{c}\;.
\end{equation*}
Equality holds if and only if the subcategory $\QSDperf{Q}{\alg}$ is thick.
\end{res*}

\noindent
A triangulated subcategory of a triangulated category is called \emph{thick} if it is closed under direct summands (see also \dfnref{compactly-generated}). In the special case of chain complexes, it is known that $\QSDperf{}{\alg}$ is a thick subcategory of $\QSD{}{\alg}$; whence equality holds in Theorem~C.  We conjecture that the same is true in general:

\begin{bfhpg*}[Conjecture]
Assume \stpref{nilpotent}. There is always an equality of triangulated subcategories of the $Q$-shaped derived category of $\alg$,
\begin{equation*}
  \QSDperf{Q}{\alg} \,=\, \QSD{Q}{\alg}^\mathrm{c}\;.
\end{equation*}
\end{bfhpg*}

\vspace*{0.5ex}

\begin{bfhpg}[Compact objects]
The notion of (strictly) perfect objects aside, another main result is the following, generalizing that the (classic) derived category is compactly generated.  This will be proved in \secref{Compact} (the objects in the set $\mathbb{S}$ are defined in \ref{recall-C-K}):
\end{bfhpg}

\begin{res*}[Theorem~D]
  Assume \stpref{nilpotent}. Then $\QSD{Q}{\alg}$ is a compactly generated triangulated category; in fact, 
\begin{equation*}  
  \mathbb{S} \,=\, \big\{\Sq{q}(\alg) \,\big|\, q \in Q \big\}
\end{equation*}  
is a set of compact generators for $\QSD{Q}{\alg}$. In particular, one has
\begin{equation*}
  \QSD{Q}{\alg}
  \,=\, 
  \operatorname{Loc}(\mathbb{S})
  \qquad \text{and} \qquad
  \QSD{Q}{\alg}^\mathrm{c}
  \,=\, 
  \operatorname{Thick}(\mathbb{S})\;.
\end{equation*}
\end{res*}

\vspace*{0.5ex}

\begin{bfhpg}[Cofibrant and fibrant objects]
Our final main result concerns the two model category structures on the functor category $\lMod{Q,\,\alg}$.  In the projective model structure on $\lMod{Q,\,\alg}$, the class of cofibrant objects is ${}^\perp\lE{}$ and in the injective model structure the class of fibrant objects is $\lE{}^\perp$, see \cite[Thm.~6.1]{HJ2021}.  The following shows that over a ring with finite global dimension, it is possible to give hands-on descriptions of theses classes:
\end{bfhpg}

\begin{res*}[Theorem~E]
Assume \stpref{nilpotent}. There are inclusions:
  \begin{align*}
    {}^\perp\lE{} 
    &\,\subseteq\, 
    \{ X \in \lMod{Q,\,\alg} \; | \; X(q) \in \lPrj{\alg} \text{ for all } q \in Q\} \quad \text{and}
    \\
    \lE{}^\perp
    &\,\subseteq\, 
    \{ X \in \lMod{Q,\,\alg} \; | \; X(q) \in \lInj{\alg} \text{ for all } q \in Q\}\;,
  \end{align*}
  and equalities hold if $\alg$ has finite left global dimension.
\end{res*}

\noindent In this theorem, $\lPrj{\alg}$ and $\lInj{\alg}$ denote the classes of projective and injective left $\alg$-modules. Theorem~E, which is proved in \secref{Fibrant}, generalizes well-known descriptions of semi-projective and semi-injective complexes over a ring with finite left global dimension; see Avramov and Foxby \cite{LLAHBF91} and \exaref{Ch-semi} for further details.  By contrast, a complete understanding of ${}^\perp\lE{}$ and $\lE{}^\perp$ is generally not available, not even for chain complexes.  An illustration of how complicated the class ${}^\perp\lE{}$ can be is found in \appref{example}.

\vspace*{1ex}

This concludes the introduction. In the next section we recap some main definitions and results from \cite{HJ2021} which are key to understanding the present paper.

\section{Prerequisites}
\label{sec:Prerequisites}

We start by recalling some general facts about (abelian) model categories.

\begin{ipg}
  \label{Localization}
Let $W$ be a collection of morphisms in a category $\mathcal{C}$. By definition, the \emph{localization} of $\mathcal{C}$ with respect to $W$ is a category
$\mathcal{C}[W^{-1}]$ together with a functor $\mathcal{C} \to \mathcal{C}[W^{-1}]$ satisfying the expected universal property; see for example Krause \cite[\S1.1]{Krause-book} or Weibel \cite[\S10.3]{Wei}. There is a standard construction of the category $\mathcal{C}[W^{-1}]$ 
and the functor $\mathcal{C} \to \mathcal{C}[W^{-1}]$ that always works if one ignores set-theoretical difficulties. In this construction, the category $\mathcal{C}[W^{-1}]$ has the same objects as $\mathcal{C}$; however, in general, the hom ``sets'' in $\mathcal{C}[W^{-1}]$ need not be small, so the ``category'' $\mathcal{C}[W^{-1}]$ might not exist within the same set-theoretic universe as $\mathcal{C}$. In the situation where $\mathcal{C}$ is a \emph{model category} and $W$ is the collection of \emph{weak equivalences} in $\mathcal{C}$ (which is the situation of interest in this paper), the localization $\mathcal{C}[W^{-1}]$ does exist; it is called the \emph{homotopy category} of $\mathcal{C}$ and denoted by $\operatorname{Ho}(\mathcal{C})$; see e.g.~Hovey \cite[\S1.2]{modcat}.
\end{ipg}

We assume that the reader has basic knowledge of \emph{cotorsion pairs}; in particular, (s)he must know what a cotorsion pair is and what it means for such a pair to be \emph{complete} or \emph{hereditary}. The book \cite{GobelTrlifaj} (Chaps.~2 and 3) by G{\"o}bel and Trlifaj is an excellent account on cotorsion pairs in module categories. Cotorsion pairs in general abelian or exact categories are treated in various papers, see e.g.~Saor{\'{\i}}n and {\v{S}}{\v{t}}ov{\'{\i}}{\v{c}}ek \cite{SaorinStovicek} and Gillespie \cite{MR2811572,MR3608936}.

\begin{ipg}
  \label{amc}
Let $\mathcal{M}$ be an abelian category. An \emph{abelian model structure} on $\mathcal{M}$ is a model structure compatible with the abelian structure in the sense of \cite[Dfn.~2.1]{Hovey02}.
By Thm.~2.2~in~\emph{loc.~cit.} any abelian model structure on $\mathcal{M}$ corresponds to a triple $(\mathcal{Q},\mathcal{W},\mathcal{R})$ of classes of objects in $\mathcal{M}$ for which $\mathcal{W}$ is thick and $(\mathcal{Q} \cap \mathcal{W},\mathcal{R})$ and $(\mathcal{Q},\mathcal{W} \cap \mathcal{R})$ are complete cotorsion pairs. Such a triple is called a
\emph{Hovey triple} in $\mathcal{M}$; in fact, $\mathcal{Q}$ is the class of cofibrant objects, $\mathcal{R}$ is the class of fibrant objects, and $\mathcal{W}$ is the class of trivial objects, i.e.~objects that are weakly equivalent to $0$. If the complete cotorsion pairs $(\mathcal{Q} \cap \mathcal{W},\mathcal{R})$ and $(\mathcal{Q},\mathcal{W} \cap \mathcal{R})$ are hereditary, then $(\mathcal{Q},\mathcal{W},\mathcal{R})$ is called a \emph{hereditary} Hovey triple and the associated abelian model structure on $\mathcal{M}$ is called a \emph{hereditary} abelian model structure. In this case, the full subcategory $\mathcal{M}_{\mathrm{cf}} = \mathcal{Q} \cap \mathcal{R}$ of cofibrant and fibrant (also called bifibrant) objects in $\mathcal{M}$ is a Frobenius exact category whose class of pro-injective objects is $\mathcal{Q}\cap \mathcal{W} \cap \mathcal{R}$, see Gillespie \cite[Prop.~4.2]{MR3608936}. The stable category \mbox{$\operatorname{Stab}(\mathcal{M}_{\mathrm{cf}}) = (\mathcal{Q} \cap \mathcal{R})\big/(\mathcal{Q}\cap \mathcal{W} \cap \mathcal{R})$} is by \cite[Thm.~4.3]{MR3608936} triangulated equivalent to the homotopy category \smash{$\operatorname{Ho}(\mathcal{M}) = \mathcal{M}[\{\mathrm{weak \ equivalences}\}^{-1}]$}; in fact, the inclusion functor $\mathcal{M}_{\mathrm{cf}} \to \mathcal{M}$ induces a triangulated equivalence,
\begin{equation}
  \label{eq:equivalence}
  \operatorname{Stab}(\mathcal{M}_{\mathrm{cf}}) \stackrel{\simeq}{\longrightarrow} \operatorname{Ho}(\mathcal{M})
\end{equation}
whose quasi-inverse is the bifibrant replacement functor $\fibr \cofr$ described in \cite[Lem.~2.5]{MR3608936}: For $X \in \mathcal{M}$ write $\cofr X$ for a cofibrant replacement and $\fibr X$ for a fibrant replacement of $X$. A bifibrant replacement of $X$ is $\fibr \cofr X$. Note that $\cofr X \in \mathcal{Q}$, $\fibr X \in \mathcal{R}$, and $\fibr \cofr X \in \mathcal{Q} \cap \mathcal{R}$.

As explained in Happel \cite[Chap.~I\S2]{Happel}, the stable category of any Frobenius category is triangulated in a canonical way, so the triangulated equivalence \eqref{equivalence} provides a rather explicit description of the triangulated structure on $\operatorname{Ho}(\mathcal{M})$. This will be useful for us.
\end{ipg}

We now give a brief summary of some definitions and results from \cite{HJ2021}, which are key to this paper.

\begin{bfhpg}[Blanket assumption]
  \label{blanket-assumptions}
Unless otherwise stated, $\Bbbk$ denotes in this paper any commutative ring, $Q$ any small $\Bbbk$-preadditive category, and $\alg$ any $\Bbbk$-algebra. 
\end{bfhpg}

\begin{bfhpg}[Notation]
  \label{not}
   We will use the notation from \cite{HJ2021}. In particular, $\lMod{\alg}$ denotes the category of left $\alg$-modules and $\lMod{Q,\,\alg}$ is the category of $\Bbbk$-linear functors $Q \to \lMod{\alg}$, see \cite[Dfn.~3.1]{HJ2021}. Further symbols along with references to where they appear in \cite{HJ2021} will follow.
\end{bfhpg}

\begin{ipg}
\label{recall-F-G}
Recall from \cite[Cor.~3.9]{HJ2021} that for every $q \in Q$ there is an adjoint triple $(\Fq{q},\Eq{q},\Gq{q})$:
  \begin{equation*}
  \xymatrix@C=4pc{
    \lMod{Q,\,\alg}
    \ar[r]^-{\Eq{q}}
    &
    \lMod{\alg}
    \ar@/_1.8pc/[l]_-{\Fq{q}}
    \ar@/^1.8pc/[l]^-{\Gq{q}}     
  }
  \qquad \text{given by} \qquad
  {\setlength\arraycolsep{1.5pt}
   \renewcommand{\arraystretch}{1.2}
  \begin{array}{rcl}
  \Fq{q}(M) &=& Q(q,-) \otimes_\Bbbk M \\
  \Eq{q}(X) &=& X\mspace{1.5mu}(q) \\ 
  \Gq{q}(M) &=& \Hom{\Bbbk}{Q(-,q)}{M}
  \end{array}
  }
  \end{equation*}
  for $M \in \lMod{\alg}$ and $X \in \lMod{Q,\,\alg}$.  Note that $\Eq{q}$ is the evaluation functor at $q$.
\end{ipg}  

Some of the definitions below are borrowed from Dell'Ambrogio, Stevenson, and  \v{S}\v{t}ov\'{\i}\-\v{c}ek \cite[Dfns. 4.1, 4.5 and Rmk. 4.7]{MR3719530}. Note that Serre functors were originally introduced by Bondal and Kapranov in \cite{BondalKapranov} (in the case where $\Bbbk$ is a field). The definitions play a key role in \cite{HJ2021} and in the present work.

\begin{ipg}
\label{Bbbk}
  We consider the following conditions on $\Bbbk$ and $Q$ (which may or may not hold).
\begin{rqm}
\item \label{eq:Homfin} \emph{Hom-finiteness}: Each hom $\Bbbk$-module $Q(p,q)$ is finitely generated and projective. 

\item \label{eq:locbd} \emph{Local boundedness}: For each $q \in Q$ there are only finitely many objects in $Q$ map\-ping nontrivially into or out of $q$, that is, the following sets are finite:
\begin{equation*}
  \hspace*{4ex}
  \operatorname{N}_-(q)=\{p \in Q \;|\; Q(p,q) \neq 0\}
  \quad \ \text{ and } \ \quad
  \operatorname{N}_+(q)=\{r \in Q \;|\; Q(q,r) \neq 0\}\;.
\end{equation*}

\item \label{eq:Serre} \emph{Existence of a Serre functor}: There exists a $\Bbbk$-linear autoequivalence $\Serre \colon Q \to Q$ and a natural isomorphism $Q(p,q) \cong \Hom{\Bbbk}{Q(q,\Serre(p))}{\Bbbk}$.

\item \label{eq:strong-retraction} \emph{Strong Retraction Property}: For $q \in Q$ the unit map $\Bbbk \to Q(q,q)$ given by $x \mapsto x\cdot\mathrm{id}_q$ has a $\Bbbk$-module retraction; whence there is a $\Bbbk$-module decomposition:
\begin{equation*}
  Q(q,q) \,=\, (\Bbbk\cdot\mathrm{id}_q) \oplus \mathfrak{r}_q\;.
\end{equation*}  
Moreover, the $\Bbbk$-submodules $\mathfrak{r}_q$ are compatible with composition in $Q$ as follows:
\begin{itemlist}
\item[($\dagger$)] $\mathfrak{r}_q \circ \mathfrak{r}_q \subseteq \mathfrak{r}_q$ for all $q \in Q$.
\item[($\ddagger$)] $Q(q,p) \circ Q(p,q) \subseteq \mathfrak{r}_p$ for all $p \neq q$ in $Q$.
\end{itemlist}

\end{rqm}
\end{ipg}

\begin{ipg}
  \label{consequences-of-srp}
  The Strong Retraction Property, i.e.~condition \eqref{strong-retraction} in \ref{Bbbk}, makes it possible to define the \emph{pseudo-radical} $\mathfrak{r}$. This important ideal in $Q$ is explained in detail in \cite[Rmk.~7.4 and Lem.~7.7]{HJ2021}. Using the pseudo-radical one can for every $q \in Q$ define \emph{stalk functors}:
\begin{equation*}
  \stalkco{q}\,=\, Q(q,-)/\mathfrak{r}(q,-) \,\in\, \lMod{Q}
  \qquad \text{and} \qquad
  \stalkcn{q} \,=\, Q(-,q)/\mathfrak{r}(-,q) \,\in\, \rMod{Q}\,,
\end{equation*}  
see \cite[Dfn.~7.9 and Lem.~7.10]{HJ2021}. Using the stalk functors one can for each $q \in Q$ and $i \geqslant 0$ define \emph{(co)homology} functors:
  \begin{equation*}
    \cH[i]{q} \,=\, \Ext{Q}{i}{\stalkco{q}}{?}
    \qquad \text{and} \qquad
    \hH[i]{q} \,=\, \Tor{Q}{i}{\stalkcn{q}}{?}\;,
  \end{equation*}
  which are functors from $\lMod{Q,\,\alg}$ to $\lMod{\alg}$; see \cite[Dfn.~7.11]{HJ2021}.  
\end{ipg}

\begin{ipg}
  \label{recall-C-K}
  Assume that $Q$ satisfies the Strong Retraction Property, i.e.~condition \eqref{strong-retraction} in \ref{Bbbk}. Recall from \cite[Prop.~7.15]{HJ2021} (see also Setup~7.13 in \emph{loc.~cit.}) that for every $q \in Q$ there is an adjoint triple $(\Cq{q},\Sq{q},\Kq{q})$:
  \begin{equation*}
  \begin{gathered}
  \xymatrix@C=4pc{
    \lMod{\alg}
    \ar[r]^-{\Sq{q}} &
    \lMod{Q,\,\alg}
    \ar@/_1.8pc/[l]_-{\Cq{q}}
    \ar@/^1.8pc/[l]^-{\Kq{q}}    
  }  \qquad \textnormal{given by} \qquad
  {\setlength\arraycolsep{1.5pt}
   \renewcommand{\arraystretch}{1.2}
  \begin{array}{rcl}
  \Cq{q}(X) &=& \ORt{\stalkcn{q}}{X} \\
  \Sq{q}(M) &=& \stalkco{q} \otimes_\Bbbk M
  \\ 
  \Kq{q}(X) &=& \Hom{Q}{\stalkco{q}}{X}
  \end{array}
  }
  \end{gathered}
  \end{equation*}
  for $M \in \lMod{\alg}$ and $X \in \lMod{Q,\,\alg}$. 
\end{ipg}

Many of the results in \cite{HJ2021} and in this paper require the following setup.

\begin{stp}
  \label{stp:nilpotent}
  The commutative ring $\Bbbk$ and the $\Bbbk$-preadditive category $Q$ from \ref{blanket-assumptions} satisfy:
 \begin{itemlist}
 \item $Q$ is Hom-finite, locally bounded, has a Serre functor, and has the Strong Retraction Property, that is, all all four conditions in \ref{Bbbk} hold.
  \item The pseudo-radical $\mathfrak{r}$ from \ref{consequences-of-srp} is nilpotent, that is, $\mathfrak{r}^N=0$ for some $N \in \mathbb{N}$.
 \item The ring $\Bbbk$ is noetherian and hereditary (e.g.~$\Bbbk$ is a field or $\Bbbk=\mathbb{Z}$).
  \item $\alg$ can be any $\Bbbk$-algebra. 
 \end{itemlist}
\end{stp} 
 
\noindent 
Let us explain the importance and the power of this setup: 
 
\begin{ipg}
  \label{summary-of-HJ2021}
  Assume \stpref{nilpotent}. In this case, there is by \cite[Dfn.~4.1]{HJ2021} an equality:
   \begin{align*}
    \lE{} &\,=\, \{ X \in \lMod{Q,\,\alg} \,|\, X^\natural \text{ has finite projective dimension in } \lMod{Q} \}
    \\
    &\,=\, \{ X \in \lMod{Q,\,\alg} \,|\, X^\natural \text{ has finite injective dimension in } \lMod{Q} \}\;,
  \end{align*}  
   where $(-)^\natural$ is the forgetful functor, see \cite[Dfn.~3.2]{HJ2021}. Objects in $\lE{}$ are called \emph{exact} and the class $\lE{}$ is by \cite[Thm.~6.1]{HJ2021} part of two different hereditary Hovey triples, 
\begin{equation*}
  (\mathcal{Q},\mathcal{W},\mathcal{R}) = ({}^\perp\lE{}, \lE{}, \lMod{Q,\,\alg})
  \qquad \text{and} \qquad
  (\mathcal{Q},\mathcal{W},\mathcal{R}) = (\lMod{Q,\,\alg}, \lE{}, \lE{}^\perp)\;,
\end{equation*}   
which, as explained in \ref{amc}, yield two different model structures on $\lMod{Q,\,\alg}$:
\begin{itemlist}   
\item The \emph{projective} model structure on $\lMod{Q,\,\alg}$ where ${}^\perp\lE{}$ is the class of cofibrant objects, $\lE{}$ is the class of trivial objects, and every object is fibrant.

\item The \emph{injective} model structure on $\lMod{Q,\,\alg}$ where $\lE{}^\perp$ is the class of fibrant objects, $\lE{}$ is the class of trivial objects, and every object is cofibrant.
\end{itemlist}
These two model structures have  by \cite[Prop.~6.3]{HJ2021} the same weak equivalences and the associated homotopy category is, as in \cite[Dfn.~6.4]{HJ2021}, denoted by
\begin{equation*}
  \QSD{Q}{\alg} \,=\, \operatorname{Ho}(\lMod{Q,\,\alg})\;.
\end{equation*}
It is called the \emph{$Q$-shaped derived category of $\alg$}. By the general theory from \ref{amc}, the category ${}^\perp\lE{}$, respectively, $\lE{}^\perp$, is a Frobenius exact category whose class of pro-injective objects is \smash{${}^\perp\lE{} \cap \lE{} = \lPrj{Q,\,\alg}$} (the projective objects in $\lMod{Q,\,\alg}$), respectively, \smash{$\lE{} \cap \lE{}^\perp = \lInj{Q,\,\alg}$} (the  injective objects in $\lMod{Q,\,\alg}$). As mentioned in \ref{amc} and recorded in \cite[Thm.~6.5]{HJ2021} there are equivalences of triangulated categories,
\begin{equation*}
  \frac{{}^\perp\mathscr{E}}{\lPrj{Q,\alg}} \ \simeq \ 
  \QSD{Q}{\alg} \ \simeq \
  \frac{\mathscr{E}^\perp}{\lInj{Q,\alg}}\;.
\end{equation*} 

Finally, the (co)homology functors \smash{$\cH[i]{q}, \hH[i]{q} \colon \lMod{Q,\,\alg} \to \lMod{\alg}$} from \ref{consequences-of-srp} detect exact objects and weak equivalences in \smash{$\lMod{Q,\,\alg}$} as described in \cite[Thms.~7.1 and 7.2]{HJ2021}. In particular, the functors $\cH[i]{q}$ and $\hH[i]{q}$ map weak equivalences in to isomorphisms, and hence they induce well-defined functors on the $Q$-shaped derived category \smash{$\QSD{Q}{\alg}$}.
\end{ipg}

Not all results in this paper require the full strength of \stpref{nilpotent}; as the reader will notice we sometimes only assume some of the conditions in \ref{Bbbk}. This is the case for e.g.~\lemref{F-G-iso} and \prpref{F-G-locally-finite}, and it is also the case for a number of results in 
\secref[Sections~]{Compact} and \secref[]{Perfect}.

\section{(Co)fibrant objects in $\lMod{Q,\,\alg}$}

\label{sec:Fibrant}

In general, the class ${}^\perp\lE{}$ of cofibrant objects in the projective model structure and the class $\lE{}^\perp$ of fibrant objects in the injective model structure on $\lMod{Q,\,\alg}$ (see \ref{summary-of-HJ2021}) are hard to perceive. The main purpose of this section is to prove Theorem~E from the Introduction, which gives hands-on descriptions of the classes ${}^\perp\lE{}$ and $\lE{}^\perp$ in the case where $\alg$ has finite global dimension. Theorem~E is also an important ingredient in the proof of Theorem~B (from the Introduction), which is established in \secref{Perfect}. Another goal in this section is \thmref{closure-properties}, which 
shows how that classes ${}^\perp\lE{}$ and $\lE{}^\perp$ behave in short exact sequences.

To parse the next definition, recall from \cite[6.4 and 6.7]{MR3990178} the definitions of \emph{filtrations} and \emph{cofiltrations} with respect to some class of objects in an abelian category.

\begin{dfn}
  For a class $\mathcal{C}$ of objects in a bicomplete abelian category $\mathcal{M}$ we set
\begin{align*}
  \Filt{\mathcal{C}} 
  &\,=\,
  \{ M \in \mathcal{M} \ | \ \textnormal{$M$ has a $\mathcal{C}$-filtration} \}\;,
  \\
  \coFilt{\mathcal{C}} 
  &\,=\,
  \{ M \in \mathcal{M} \ | \ \textnormal{$M$ has a $\mathcal{C}$-cofiltration} \}\;.
\end{align*}  
\end{dfn}

This definition is important in the next result, and so are the (stalk) functors $\Sq{q}$ from~\ref{recall-C-K}. Furthermore, 
$\lPrj{\alg}$ and $\lInj{\alg}$ denote the classes of projective and injective left $\alg$-modules.

\begin{prp}
  \label{prp:E-perp-inclusions}
Assume \stpref{nilpotent}. The following assertions hold.
 \begin{prt}
 \item
 With $\mathcal{P} = \{ \Sq{q}(P) \ | \ q \in Q \textnormal{ and } P \in \lPrj{\alg} \}$ there are inclusions,
  \begin{equation*}
    \Filt{\mathcal{P}} \ \subseteq \
    {}^\perp\lE{} 
    \ \subseteq \
    \{ X \in \lMod{Q,\,\alg} \; | \; X(q) \in \lPrj{\alg} \text{ for all } q \in Q \}\;.
  \end{equation*}
  
  \item With $\mathcal{I} = \{ \Sq{q}(I) \ | \ q \in Q \textnormal{ and } I \in \lInj{\alg} \}$ there are inclusions,
  \begin{equation*}
    \hspace*{-2ex}
    \coFilt{\mathcal{I}} \ \subseteq \   
    \lE{}^\perp
    \ \subseteq \
    \{ X \in \lMod{Q,\,\alg} \; | \; X(q) \in \lInj{\alg} \text{ for all } q \in Q \}\;.
  \end{equation*}
  \end{prt}
\end{prp}

\begin{proof}
  \proofoftag{a} To prove the first inclusion, it suffices by Eklof's lemma \cite[Lem.~6.6]{MR3990178} to show that $\mathcal{P}$ is a subset of ${}^\perp\lE{}$. Let $q \in Q$ and $P \in \lPrj{\alg}$ be given. Recall from \ref{recall-C-K} that the functor $\Sq{q}$ is exact with right adjoint $\Kq{q}$. The first right derived functor $\mathrm{R}^1\Kq{q}$ is the cohomology functor functor $\cH[1]{q}$ from \ref{consequences-of-srp}. Thus, for every $E \in \lE{}$ one has \smash{$\mathrm{R}^1\Kq{q}(E) = \cH[1]{q}(E)=0$} by \cite[Thm.~7.1]{HJ2021}, and the first isomorphism below now follows from \cite[Lem.~1.3]{MR4013804}. The second isomorphism holds as $P \in \lPrj{\alg}$.
  \begin{equation*}
    \Ext{Q,\,\alg}{1}{\Sq{q}(P)}{E} \,\cong\,
    \Ext{\alg}{1}{P}{\Kq{q}(E)} \,\cong\, 0\;.
  \end{equation*}
This proves that $\Sq{q}(P)$ belongs to ${}^\perp\lE{}$, and hence $\mathcal{P} \subseteq {}^\perp\lE{}$, as desired.

To prove the second inclusion, let $X \in {}^\perp\lE{}$ and $q \in Q$ be given. The evaluation functor $\Eq{q}$ at $q$ has by \ref{recall-F-G} a right adjoint $\Gq{q}$, so there is a natural isomorphism,
\begin{equation*}
  \Hom{\alg}{X(q)}{-} \,=\, \Hom{\alg}{\Eq{q}(X)}{-}
  \,\cong\, \Hom{Q,\,\alg}{X}{\Gq{q}(-)}\;.
\end{equation*}
Thus, we must prove that the last functor above is exact. Let $\xi = 0 \to M' \to M \to M'' \to 0$ be an exact sequence in $\lMod{\alg}$. As the category $Q$ is assumed to be Hom-finite, the functor $\Gq{q}$ is exact by \cite[Cor.~3.9(b)]{HJ2021} and hence $\Gq{q}(\xi)$ is an exact sequence in $\lMod{Q,\,\alg}$. To prove exactness of the sequence $\Hom{Q,\,\alg}{X}{\Gq{q}(\xi)}$, it suffices to argue that one has
\begin{equation}
  \label{eq:Ext-is-0} 
  \Ext{Q,\,\alg}{1}{X}{\Gq{q}(M')}\,=\,0\;. 
\end{equation}  
By the the assumptions on $Q$, we can apply \cite[Thm.~7.1 and Lem.~7.14(b)]{HJ2021} to conclude that $\Gq{q}(M') \in\lE{}$; and since one has $X \in {}^\perp\lE{}$, it follows that \eqref{Ext-is-0} holds.

\proofoftag{b} Dual to the proof of part \prtlbl{a}. Note that instead of Eklof's lemma \cite[Lem.~6.6]{MR3990178} and \cite[Lem.~1.3]{MR4013804} one uses Trlifaj's lemma \cite[Lem.~6.8]{MR3990178} and \cite[Lem.~1.2]{MR4013804}.
\end{proof}

We are now in a position to prove Theorem~E from the Introduction.

\begin{proof}[Proof of Theorem~E]
  The inclusions were established in \prpref{E-perp-inclusions}. Now assume that $\alg$ has finite left global dimension. We only prove that the first inclusion is an equality; a similar argument shows that the second inclusion is an equality. Let  $X \in \lMod{Q,\,\alg}$ be an object that satisfies $X(q) \in \lPrj{\alg}$ for all $q \in Q$. We must argue that $X$ belongs to ${}^\perp\lE{}$.
Since $({}^\perp\lE{},\lE{})$ is a complete cotorsion pair in $\lMod{Q,\,\alg}$, see \ref{summary-of-HJ2021} (or \cite[Thm.~5.5]{HJ2021}), there exists an exact sequence,
\begin{equation}
  \label{eq:EPX}
  0 \longrightarrow E \longrightarrow P \longrightarrow X \longrightarrow 0
\end{equation}  
with $P \in {}^\perp\lE{}$ and $E \in \lE{}$. The goal is to show that this sequence splits; this will imply that $X \in {}^\perp\lE{}$ since $P \in {}^\perp\lE{}$. We have $P(q),X(q) \in \lPrj{\alg}$ for all $q \in Q$ by \prpref{E-perp-inclusions}\prtlbl{a} and by assumption. As the sequence $0 \to E(q) \to P(q) \to X(q) \to 0$ is exact, it follows that also $E(q) \in \lPrj{\alg}$ for all $q \in Q$. 

At this point, we pause to make a general observation: Let $n \in \mathbb{N}_0$ be the left global dimension of $\alg$. Evidently, the category  $\lMod{\alg}$ is locally $n$-Gorenstein in the sense of \cite[Dfns.~2.1 and 2.5]{MR3719530}, so by definition (see the remarks above \cite[Thm.~4.6]{MR3719530}), this means that $\alg$ is an $n$-Gorenstein algebra. Hence \cite[Cor.~4.8]{MR3719530}\footnote{Note that \cite[Cor.~4.8]{MR3719530} deals with $\alg \otimes Q$-modules $M$, i.e.~$\Bbbk$-linear functors $M \colon \alg \otimes Q \to \lMod{\Bbbk}$, where $\alg \otimes Q$ is the \emph{extension} of $Q$ by $\alg$ defined above \cite[Thm.~4.6]{MR3719530}. However, it is easy to see that the category $\lMod{\alg \otimes Q}$ of $\alg \otimes Q$-modules is equivalent to the category $\lMod{Q,\,\alg}$ of $\Bbbk$-linear functors $Q \to \lMod{\alg}$.}
 implies that an object $M \in \lMod{Q,\,\alg}$~is Gorenstein projective if and only if the $\alg$-module $M(q)$ is Gorenstein projective for every $q \in Q$. Since $\alg$ has finite left global dimension, the latter condition just means that $M(q) \in \lPrj{\alg}$ for every $q \in Q$, see \cite[Prop.~10.2.3]{rha}.
 
By the observation above, it follows that $X$, $P$, and $E$ are Gorenstein projective objects in $\lMod{Q,\,\alg}$. We will now show that $E$ is even a projective object in $\lMod{Q,\,\alg}$. Once this has been shown, it follows immediately from the definition of Gorenstein projective objects, see \cite[Dfn.~2.20]{MR2404296}, that $\Ext{Q,\,\alg}{1}{X}{E}=0$, and thus \eqref{EPX} will split. We already know that $E \in \lE{}$. Thus, to show that $E$ is projective, we need by \cite[Thms.~7.1 and 7.29]{HJ2021} to argue that the $\alg$-module $\Cq{q}(E)$ is projective for every $q \in Q$, where $\Cq{q}$ is the functor from \ref{recall-C-K}. As $E$ is a Gorenstein projective object in $\lMod{Q,\,\alg}$ there is, by definition, an exact sequence,
\begin{equation}
  \label{eq:Pi}
  0 \longrightarrow E \longrightarrow P^0 \longrightarrow \cdots \longrightarrow P^{n-1} \longrightarrow E' \longrightarrow 0
\end{equation}  
in $\lMod{Q,\,\alg}$ where each $P^i$ is projective (and $E'$ is Gorenstein projective, but that is not important here). By another application of \cite[Thms.~7.1 and 7.29]{HJ2021}, each $P^i$ belongs to $\lE{}$, and hence $E'$ belongs to $\lE{}$ by the last assertion in \cite[Thm.~4.4]{HJ2021}. Now fix $q \in Q$. The right exact functor $\Cq{q}$ is, in general, not exact; indeed, its $i^\mathrm{th}$ left derived functor $\mathrm{L}_i\Cq{q}$ is the homology functor $\hH[i]{q}$ from \ref{consequences-of-srp}. Since $E'$ is in $\lE{}$ we have $\mathrm{L}_i\Cq{q}(E')=\hH[i]{q}(E')=0$ for all $i>0$ by \cite[Thm.~7.1]{HJ2021}, and thus \eqref{Pi} does, in fact, induce an exact sequence,
\begin{equation*}
 0 \longrightarrow \Cq{q}(E) \longrightarrow \Cq{q}(P^0) \longrightarrow \cdots \longrightarrow \Cq{q}(P^{n-1}) \longrightarrow \Cq{q}(E') \longrightarrow 0\;,
\end{equation*}
of left $\alg$-modules. Each module $\Cq{q}(P^i)$ is projective by \cite[Thm.~7.29]{HJ2021}, and the module $\Cq{q}(E')$ has projective dimension at most $n$ (= the left global dimension of $\alg$). The exact sequence therefore shows that $\Cq{q}(E)$ is projective, as desired.
\end{proof}

\begin{exa}
  \label{exa:Ch-semi}
  Consider the situation where $Q$ is the mesh category of the repetitive quiver of \smash{$\vec{\mathbb{A}}_2 \,=\, \vtx \to \vtx$}\,; in other words, $Q$ is the path category of the quiver \eqref{Gamma} with the relations $\partial^2 = 0$. In this case,
  $\lMod{Q,\,\alg}$ is equivalent to the category $\operatorname{Ch}\mspace{1mu}\alg$ of (chain) complexes of left $\alg$-modules, see \cite[Rmk.~8.5 and Exa.~8.13]{HJ2021}. For the class
$\lE{}$ one has:
\begin{align*}
  \lE{}
  &\,=\,
  \{ \textnormal{exact (or acyclic) complexes} \}\;,
  \\
  {}^\perp\lE{}
  &\,=\,
  \{ \textnormal{semi-projective (or DG-projective) complexes} \}\;,  \quad \textnormal{and}
  \\
  \lE{}^\perp
  &\,=\,
  \{ \textnormal{semi-injective (or DG-injective) complexes} \}\;.
\end{align*}  
Indeed, the first equality follows easily from \cite[Thm.~7.1]{HJ2021} and the other two follow from Garc{\'{\i}}a Rozas \cite[Props.~2.3.4 and 2.3.5]{GR99}. 

In this case, the class $\mathcal{P}$ from \prpref{E-perp-inclusions} is just $\{ \upSigma^qP \ | \ q \in \mathbb{Z} \textnormal{ and } P \in \lPrj{\alg} \}$, where $\upSigma$ is the shift functor on $\operatorname{Ch}\mspace{1mu}\alg$. It is easily seen that every bounded below complex of projective $\alg$-modules, $\cpx{P} \,=\ \cdots \to P_{u+2} \to P_{u+1} \to P_{u} \to 0 \to 0 \to 0 \to \cdots$, is in  $\Filt{\mathcal{P}}$. Hence every such complex is semi-projective by the first inclusion in part \prtlbl{a} of \prpref{E-perp-inclusions}. Similarly, the first inclusion in part \prtlbl{b} of the same result shows that every bounded above complex of injective $\alg$-modules, $\cpx{I} \,=\ \cdots \to 0 \to 0 \to 0 \to I_{v} \to I_{v-1} \to I_{v-2} \to \cdots$, is semi-injective. This is of couse well-known. The second inclusions in parts \prtlbl{a} and \prtlbl{b} of \prpref{E-perp-inclusions} assert that every semi-projective/injective complex must consist of projective/injective modules, which is also well-known. Theorem~E in the Introduction shows that over a ring with finite global dimension, a semi-projective/injective complex is \textsl{nothing but} a complex consisting of projective/injective modules. This special case of Theorem~E can be found in \cite[Prop.~3.4]{LLAHBF91} by Avramov and Foxby. Applications of their result (and related results) can be found in \cite{Iacob06,MR2568347} by Iacob and Iyengar.  
\end{exa}

An important fact, which we shall now prove, is that when $Q$ is Hom-finite and has a Serre functor, the family $\{\Fq{q}\}_{q \in Q}$ is just a permutation of the family $\{\Gq{q}\}_{q \in Q}$, see \ref{recall-F-G}.

\begin{lem}
  \label{lem:F-G-iso}
  Assume that $Q$ is Hom-finite and has a Serre functor, i.e.~conditions \eqref{Homfin} and \eqref{Serre} in \ref{Bbbk} hold. For every $q \in Q$ there is a natural isomorphism of functors $\Gq{\,\Serre(q)} \cong \Fq{q}$.
\end{lem}

\begin{proof}
  Let $q \in Q$ be given. For every $M \in \lMod{\alg}$ and $p \in Q$ one has the following isomorphisms of left $\alg$-modules:
  \begin{align*}
    \Gq{\,\Serre(q)}(M)(p) 
    &\,\cong\,
    \Hom{\Bbbk}{Q(p,\Serre(q))}{M}
    \\
    &\,\cong\,
    \Hom{\Bbbk}{Q(p,\Serre(q))}{\Bbbk} \otimes_\Bbbk M
    \\
    &\,\cong\,
    Q(q,p) \otimes_\Bbbk M
    \\
    &\,\cong\,
    \Fq{q}(M)(p)\;.
  \end{align*}
  In this computation, the first and last isomorphisms hold by the definitions of the functors $\Gq{\,\Serre(q)}$ and $\Fq{q}$, see \ref{recall-F-G}. The second isomorphism is induced by the so-called tensor evaluation map, which in this case is an isomorphism as the $\Bbbk$-module $Q(p,\Serre(q))$ is finitely generated and projective by Hom-finiteness of $Q$. This isomorphism can easily be proved but it can also be found in e.g.~\cite[Lem.~1.1]{TIs65}. The third isomorphism holds by the definition of the Serre functor. All the isomorphisms above are natural in $p$ and $M$.
\end{proof}

\begin{dfn}
  A family $\{X_i\}_{i \in I}$ of objects in $\lMod{Q,\,\alg}$ is called \emph{locally finite} if for each $q \in Q$ the set $\{i \in I \,|\, X_i(q) \neq 0\}$ is finite. 
\end{dfn}

\begin{lem}
  \label{lem:locally-finite}
  If $\{X_i\}_{i \in I}$ is a locally finite family of objects in $\lMod{Q,\,\alg}$, then there is an isomorphism $\bigoplus_{i \in I}X_i \cong \prod_{i \in I}X_i$.
\end{lem}

\begin{proof}
  Straightforward.
\end{proof}

\begin{prp}
  \label{prp:F-G-locally-finite}
  Assume that $Q$ is locally bounded, that is, condition \eqref{locbd} in \ref{Bbbk} holds. For every family $\{M_q\}_{q \in Q}$ of objects in $\lMod{\alg}$, there are isomorphisms:
\begin{equation*}
  \textstyle
  \bigoplus_{q \in Q} \Fq{q}(M_q) \,\cong\,
  \prod_{q \in Q} \Fq{q}(M_q)
  \qquad \text{and} \qquad
  \bigoplus_{q \in Q} \Gq{q}(M_q) \,\cong\,  
  \prod_{q \in Q} \Gq{q}(M_q)\;.
\end{equation*}  
\end{prp}

\begin{proof}
  By \lemref{locally-finite} it suffices to argue that the families
$\{\Fq{q}(M_q)\}_{q \in Q}$ and $\{\Gq{q}(M_q)\}_{q \in Q}$ are locally finite. We only consider the first family. It must be proved that for every $r \in Q$ the set $\{q \in Q \,|\, \Fq{q}(M_q)(r) \neq 0\}$ is finite. By definition one has $\Fq{q}(M_q)(r) = Q(q,r) \otimes_\Bbbk M_q$, see~\ref{recall-F-G}, so the desired conlusion follows from the local boundedness of $Q$, which yields that the set $  \operatorname{N}_-(r)=\{q \in Q \;|\; Q(q,r) \neq 0\}$ is finite. 
\end{proof}

\begin{dfn}
  \label{dfn:objectwise-split}
  A short exact sequence $\xi \,=\, 0 \to X' \to X \to X'' \to 0$ in $\lMod{Q,\,\alg}$ is said to be \emph{objectwise split} if the sequence $\Eq{q}(\xi) \,=\, \xi(q) \,=\, 0 \to X'(q) \to X(q) \to X''(q) \to 0$ splits in $\lMod{\alg}$ for every object $q \in Q$.
\end{dfn}

We end by proving \thmref{closure-properties} announced in the beginning of this section. This result plays an important role in the proof of \prpref{conflation}, which yields certain canonical conflations in the exact categories ${}^\perp\lE{}$ and $\lE{}^\perp$.

\begin{thm}
  \label{thm:closure-properties}
  Assume that \stpref{nilpotent} holds and let $\xi \,=\, 0 \to X' \to X \to X'' \to 0$ be a short exact sequence in $\lMod{Q,\,\alg}$. For the class ${}^\perp\lE{}$ the following assertions hold:
  \begin{prt}
  \item Assume that $X'' \in {}^\perp\lE{}$. One has $X' \in {}^\perp\lE{}$ if and only if $X \in {}^\perp\lE{}$.
  \item Assume that $\xi$ is objectwise split. If one has $X', X \in {}^\perp\lE{}$, then $X'' \in {}^\perp\lE{}$.  
  \end{prt}
Dually, for the class $\lE{}^\perp$ the following assertions hold:
  \begin{prt}
  \setcounter{prt}{2}
  \item Assume that $X' \in \lE{}^\perp$. One one has $X \in \lE{}^\perp$ if and only if $X'' \in \lE{}^\perp$.
  \item Assume that $\xi$ is objectwise split. If one has $X, X'' \in \lE{}^\perp$, then $X' \in \lE{}^\perp$.  
  \end{prt}  
\end{thm}

\begin{proof}
  Parts \prtlbl{a} and \prtlbl{c} follow immediately from the fact that the cotorsion pairs $({}^\perp\lE{},\lE{})$ and $(\lE{},\lE{}^\perp)$ are hereditary, see \ref{summary-of-HJ2021} (or \cite[Thm.~4.4]{HJ2021}). It remains to show parts \prtlbl{b} and \prtlbl{d}; we only prove \prtlbl{b} as the proof of \prtlbl{d} is dual. Write $\alpha$ for the monomorphism $X' \to X$ in the given exact sequence $\xi$. By assumption, $\alpha(q)$ is split monic in $\lMod{\alg}$ for every $q \in Q$ and the objects $X'$ and $X$ belong to ${}^\perp\lE{}$. It must be shown that $\Ext{Q,\,\alg}{1}{X''}{E}=0$ holds for every $E \in \lE{}$. Fix $E \in \lE{}$. Application of the functor $\Hom{Q,\,\alg}{-}{E}$ to the sequence $\xi$ yields an exact sequence,
\begin{equation*}
  \xymatrix@C=1.7pc{
    \Hom{Q,\,\alg}{X}{E} \ar[rr]^-{\Hom{Q,\,\alg}{\alpha}{E}} & &
    \Hom{Q,\,\alg}{X'}{E} \ar[r] &
    \Ext{Q,\,\alg}{1}{X''}{E} \ar[r] &
    \Ext{Q,\,\alg}{1}{X}{E} = 0
  },
\end{equation*}  
so it suffices to argue that the homomorphism $\Hom{Q,\,\alg}{\alpha}{E}$ is surjective. The category $\lMod{Q,\,\alg}$ has enough projectives, see e.g.~\cite[Prop.~3.12]{HJ2021}, so there exists an  exact sequence \smash{$0 \to E' \to P \xrightarrow{\pi} E \to 0$} with $P$ projective. As $P \in \lPrj{Q,\,\alg} \subseteq \lE{}$ it follows from \cite[Thm.~4.4]{HJ2021} that $E' \in \lE{}$. Application of $\Hom{Q,\,\alg}{X'}{-}$ to this sequence yields an exact sequence,
\begin{equation*}
  \xymatrix@C=1.7pc{
    \Hom{Q,\,\alg}{X'}{P} \ar[rr]^-{\Hom{Q,\,\alg}{X'}{\pi}} & &
    \Hom{Q,\,\alg}{X'}{E} \ar[r] &
    \Ext{Q,\,\alg}{1}{X'}{E'} = 0
  },
\end{equation*}  
so the map $\Hom{Q,\,\alg}{X'}{\pi}$ is surjective. In view of this and of the commutative diagram
\begin{equation*}
  \xymatrix@C=4pc{
    \Hom{Q,\,\alg}{X}{P} 
    \ar[d]_-{\Hom{Q,\,\alg}{\alpha}{P}}
    \ar[r]^-{\Hom{Q,\,\alg}{X}{\,\pi}} & 
    \Hom{Q,\,\alg}{X}{E}
    \ar[d]^-{\Hom{Q,\,\alg}{\alpha}{E}}
    \\
    \Hom{Q,\,\alg}{X'}{P} 
    \ar@{->>}[r]^-{\Hom{Q,\,\alg}{X'}{\,\pi}} & 
    \Hom{Q,\,\alg}{X'}{E}\;,\mspace{-8mu}
  }
\end{equation*}  
surjectivity of $\Hom{Q,\,\alg}{\alpha}{E}$ will follow if we can prove surjectivity of $\Hom{Q,\,\alg}{\alpha}{P}$. Recall from \cite[Prop.~3.12(a)]{HJ2021} that $\{\Fq{q}(\alg)\}_{q \in Q}$ is a family of projective generators of $\lMod{Q,\,\alg}$, and hence the projective object $P$ is a direct summand in an object of the form $\bigoplus_{q \in Q}\Fq{q}(\alg)^{(I_q)}$ for suitable index sets $I_q$. Note that one has
\begin{equation*}
  \textstyle
  \bigoplus_{q \in Q}\Fq{q}(\alg)^{(I_q)}
  \,\cong\,
  \bigoplus_{q \in Q}\Fq{q}(\alg^{(I_q)})
  \,\cong\,
  \prod_{q \in Q}\Gq{\,\Serre(q)}(\alg^{(I_q)})\;.
\end{equation*}
The first isomorphism holds as each functor $\Fq{q}$ preserves coproducts; indeed, $\Fq{q}$ is a left adjoint by \ref{recall-F-G}. The second isomorphism follows from \lemref{F-G-iso}
 and \prpref{F-G-locally-finite}.
These considerations, and the fact that the functor $\Hom{Q,\,\alg}{\alpha}{-}$ preserves products, show that without loss of generality we may assume that $P$ has the form \smash{$P=\Gq{p}(\alg^{(I)})$} for some object $p \in Q$ and index set $I$. Set $M = \alg^{(I)}$ and note that one has
\begin{equation*}
  \Hom{Q,\,\alg}{\alpha}{P}
  \,=\,
  \Hom{Q,\,\alg}{\alpha}{\Gq{p}(M)}
  \,\cong\,
  \Hom{\alg}{\alpha(p)}{M}
\end{equation*}
since $\Gq{p}$ is right adjoint of the evaluation functor at $p$, see \ref{recall-F-G}. By assumption, $\alpha(p)$ is split monic, and hence $\Hom{\alg}{\alpha(p)}{M}$ is (split) surjective, as desired.  
\end{proof}

\begin{rmk}
  Easy examples show that the conclusions in parts \prtlbl{b} and \prtlbl{d} in \thmref{closure-properties} above fail is the exact sequence $\xi$ is not objectwise split.
\end{rmk}

\section{Compact objects in $\QSD{Q}{\alg}$}

\label{sec:Compact}

The purpose of this section is to prove Theorem~D from the Introduction. We will obtain this theorem as a special case of 
\thmref{HoM}, which is akin to Hovey's \cite[Cor.~7.4.4]{modcat}.

We begin with a familiar looking Ext-formula for the hom sets in the homotopy category of an abelian model category. This formula can certainly be found in various special cases in the literature. Even in the generality below, the formula is probably known to the experts; however, since we were not able to find a reference, we have included a proof. To parse the result, recall from \ref{amc} that $\mathscr{R}$ and $\mathscr{Q}$ denote the fibrant and cofibrant replacement functors.

\begin{prp}
  \label{prp:Ext-formula}
  Let $\mathcal{M}$ be an abelian category equipped with a hereditary abelian model structure. Consider the associated triangulated homotopy category $\operatorname{Ho}(\mathcal{M})$ whose translation functor we denote by $\upSigma$. For every $X,Y \in \mathcal{M}$ and $n>0$ there are isomorphisms:
  \begin{equation*}
    \Hom{\operatorname{Ho}(\mathcal{M})}{\upSigma^{-n}X}{Y}
    \,\cong\,
    \Ext{\mathcal{M}}{n}{\fibr \cofr X}{\fibr \cofr Y}
    \,\cong\,
    \Hom{\operatorname{Ho}(\mathcal{M})}{X}{\upSigma^nY}\;.
  \end{equation*}
\end{prp}

\begin{proof}
  Since $\upSigma$ is an autoequivalence on $\operatorname{Ho}(\mathcal{M})$, the abelian groups $\Hom{\operatorname{Ho}(\mathcal{M})}{\upSigma^{-n}X}{Y}$ and $\Hom{\operatorname{Ho}(\mathcal{M})}{X}{\upSigma^nY}$ are isomorphic, so we only need to prove the second of the asserted isomorphisms. Let $(\mathcal{Q},\mathcal{W},\mathcal{R})$ be the hereditary Hovey triple that corresponds to the given hereditary abelian model structure on $\mathcal{M}$ via \cite[Thm.~2.2]{Hovey02}. As explained in \ref{amc}, the~bi\-fibrant replacement functor $\fibr \cofr$ yields a triangulated equivalence $\operatorname{Ho}(\mathcal{M}) \to \operatorname{Stab}(\mathcal{M}_{\mathrm{cf}})$. We will use the standard notation $\underline{\operatorname{Hom}}_{\mathcal{M}}$ for the hom-sets in the stable category, that is, for bifibrant objects $A$ and $B$ in $\mathcal{M}$ one has
\begin{equation}
  \label{eq:uHom}
  \uHom{\mathcal{M}}{A}{B}
  \,=\, \Hom{\mathcal{M}}{A}{B}\big/\mathcal{I}(A,B)\;,
\end{equation}  
where $\mathcal{I}(A,B)$ is the subgroup of $\Hom{\mathcal{M}}{A}{B}$ consisting of morphisms $A \to B$ in $\mathcal{M}$ that factor through some object in $\mathcal{Q}\cap \mathcal{W} \cap \mathcal{R}$. As $\fibr \cofr \colon \operatorname{Ho}(\mathcal{M}) \to \operatorname{Stab}(\mathcal{M}_{\mathrm{cf}})$ is a triangulated equivalence there is an isomorphism,
\begin{equation*}
  \Hom{\operatorname{Ho}(\mathcal{M})}{X}{\upSigma^nY}
  \,\cong\,
  \uHom{\mathcal{M}}{\fibr \cofr X}{\upSigma^n \fibr \cofr Y}\;,
\end{equation*}
for every $X, Y \in \mathcal{M}$ and $n \in \mathbb{Z}$. In view of this isomorphism, the
second isomorphism in the Proposition will follow if we can establish an isomorphism
\begin{equation}
  \label{eq:Extn-iso}
  \uHom{\mathcal{M}}{A}{\upSigma^nB}
  \,\cong\,
  \Ext{\mathcal{M}}{n}{A}{B}
\end{equation}  
for all objects $A, B \in \mathcal{Q}\cap \mathcal{R}$ (for example $A=\fibr \cofr X$ and $B=\fibr \cofr Y$) and $n>0$.

To establish the isomorphism \eqref{Extn-iso}, we start with the case $n=1$. By definition of the translation functor $\upSigma$ in the stable category of a Frobenius category, see Happel \cite[Chap. I\S2]{Happel}, the (translated) object $\upSigma B \in \mathcal{Q}\cap \mathcal{R}$ fits into a short exact sequence in $\mathcal{M}$,
\begin{equation}
  \label{eq:BB}
  \xymatrix@C=1.5pc{
    0 \ar[r] & B \ar[r] & W \ar[r]^-{\theta} & \upSigma B \ar[r] & 0
  },
\end{equation}
where $W$ is in $\mathcal{Q}\cap \mathcal{W} \cap \mathcal{R}$. Application of the functor $\Hom{\mathcal{M}}{A}{-}$ to this sequence yields the exact sequence
\begin{equation*}
  \xymatrix@C=1.5pc{
    \Hom{\mathcal{M}}{A}{W} \ar[rr]^-{\Hom{\mathcal{M}}{A}{\theta}} & & \Hom{\mathcal{M}}{A}{\upSigma B} \ar[r] & \Ext{\mathcal{M}}{1}{A}{B} \ar[r] & \Ext{\mathcal{M}}{1}{A}{W}=0
  };
\end{equation*}
here $\Ext{\mathcal{M}}{1}{A}{W}$ is zero as $(\mathcal{Q}, \mathcal{W} \cap \mathcal{R})$ is a cotorsion pair, $A$ is in $\mathcal{Q} \cap \mathcal{R} \subseteq \mathcal{Q}$, and $W$ is in $\mathcal{Q}\cap \mathcal{W} \cap \mathcal{R} \subseteq \mathcal{W} \cap \mathcal{R}$. Thus,  $\Ext{\mathcal{M}}{1}{A}{B}$ is the cokernel of $\Hom{\mathcal{M}}{A}{\theta}$ and the isomorphism \eqref{Extn-iso}, still for $n=1$, will follow if we can show the equality
\begin{equation*}
  \Im{\Hom{\mathcal{M}}{A}{\theta}} \,=\, \mathcal{I}(A,\upSigma B)\;.
\end{equation*}
The inclusion ``$\subseteq$'' is clear as $W$ belongs to $\mathcal{Q}\cap \mathcal{W} \cap \mathcal{R}$. Conversely, let $\alpha \colon A \to \upSigma B$ be any morphism that factors through an object in $\mathcal{Q}\cap \mathcal{W} \cap \mathcal{R}$, i.e.~there is some factorization,
\begin{equation*}
  \xymatrix@!=0.7pc{
    A \ar@{.>}[dr]_-{\varphi} \ar[rr]^-{\alpha} & & \upSigma B
    \\
    & V \ar@{.>}[ur]_-{\psi} & 
  }
\end{equation*}
with $V \in \mathcal{Q}\cap \mathcal{W} \cap \mathcal{R}$. Application of the functor to $\Hom{\mathcal{M}}{V}{-}$ to the short exact sequence \eqref{BB} yields an exact sequence
\begin{equation*}
  \xymatrix@C=1.5pc{
    \Hom{\mathcal{M}}{V}{W} \ar[rr]^-{\Hom{\mathcal{M}}{V}{\theta}} & & \Hom{\mathcal{M}}{V}{\upSigma B} \ar[r] & \Ext{\mathcal{M}}{1}{V}{B} = 0 
  };
\end{equation*}
here $\Ext{\mathcal{M}}{1}{V}{B}$ is zero since $(\mathcal{Q} \cap \mathcal{W},\mathcal{R})$ is a cotorsion pair, $V$ is in $\mathcal{Q}\cap \mathcal{W} \cap \mathcal{R} \subseteq \mathcal{Q} \cap \mathcal{W}$, and $B$ is in $\mathcal{Q} \cap \mathcal{R} \subseteq \mathcal{R}$. Thus the homomorphism $\Hom{\mathcal{M}}{V}{\theta}$ is surjective, so there exists a morphism $\omega \colon V \to W$ in $\mathcal{M}$ with $\theta\omega = \psi$. Therefore $\alpha = \psi\varphi = \theta\omega\varphi = \Hom{\mathcal{M}}{A}{\theta}(\omega\varphi)$. This proves the other inclusion ``$\supseteq$'' and  hence the isomorphism in \eqref{Extn-iso} is proved for $n=1$.

Now assume, by induction, that for some $n>0$ an isomorphism \eqref{Extn-iso} exists for every $A,B \in \mathcal{Q} \cap \mathcal{R}$. As also $\upSigma B$ belongs to $\mathcal{Q} \cap \mathcal{R}$, the induction hypothesis yield:
\begin{equation}
  \label{eq:induction}
  \uHom{\mathcal{M}}{A}{\upSigma^{n+1}B}
  \,\cong\,
  \uHom{\mathcal{M}}{A}{\upSigma^{n}\upSigma B}
  \,\cong\,
  \Ext{\mathcal{M}}{n}{A}{\upSigma B}\;.
\end{equation} 
Moreover, by applying $\Hom{\mathcal{M}}{A}{-}$ to the short sequence \eqref{BB} we get an exact sequence
\begin{equation*}
  \xymatrix@C=1.5pc{
    0 = \Ext{\mathcal{M}}{n}{A}{W} \ar[r]
    & 
    \Ext{\mathcal{M}}{n}{A}{\upSigma B} \ar[r]
    & 
    \Ext{\mathcal{M}}{n+1}{A}{B} \ar[r]
    & 
    \Ext{\mathcal{M}}{n+1}{A}{W}=0
  };
\end{equation*}
note that $\Ext{\mathcal{M}}{i}{A}{W}=0$ for all $i \geqslant 1$ as
$(\mathcal{Q}, \mathcal{W} \cap \mathcal{R})$ is a \textsl{hereditary} cotorsion pair, $A$ is in $\mathcal{Q} \cap \mathcal{R} \subseteq \mathcal{Q}$, and $W$ is in $\mathcal{Q}\cap \mathcal{W} \cap \mathcal{R} \subseteq \mathcal{W} \cap \mathcal{R}$.  Hence $\Ext{\mathcal{M}}{n}{A}{\upSigma B} \cong \Ext{\mathcal{M}}{n+1}{A}{B}$, which in combination with   \eqref{induction} yields an isomorphism $\uHom{\mathcal{M}}{A}{\upSigma^{n+1}B} \cong \Ext{\mathcal{M}}{n+1}{A}{B}$.
\end{proof}

Next we recall some standard definitions.

\begin{dfn} 
  \label{dfn:compact}
  Let $\mathcal{A}$ be an additive category (e.g.~an abelian or a triangulated category) with set-indexed coproducts. An object $C \in \mathcal{A}$ is called \emph{compact} if the functor $\Hom{\mathcal{A}}{C}{-}$ preserves coproducts, that is, if the canonical map
  \begin{equation*}
    \textstyle
    \bigoplus_{i \in I}\Hom{\mathcal{A}}{C}{X_i}
    \longrightarrow
    \Hom{\mathcal{A}}{C}{\bigoplus_{i \in I}X_i}
  \end{equation*}
  is an isomorphism for every (set-indexed) family $\{X_i\}_{i \in I}$ of objects in $\mathcal{A}$. We set
\begin{equation*}
  \mathcal{A}^\mathrm{c} \,=\,
  \big\{ C \in \mathcal{A} \ \big| \ \text{$C$ is compact} \big\}\;.  
\end{equation*}
\end{dfn}

\begin{dfn}
  \label{dfn:compactly-generated}
  Let $\mathcal{T}$ be a triangulated category with set-indexed coproducts.
\begin{itemlist}
  \item One says that $\mathcal{T}$ is \emph{compactly generated} if there exists a set (as opposed to a proper class) $\mathbb{S}$ of compact objects in $\mathcal{T}$ with the property that the only object $X \in \mathcal{T}$ that satisfies $\Hom{\mathcal{T}}{S}{X}=0$ for every $S \in \mathbb{S}$ is $X = 0$. See Neeman \cite[Dfn.~1.7]{MR1308405}.
\end{itemlist}
Next, let $\mathcal{S}$ be a triangulated subcategory of $\mathcal{T}$.  
\begin{itemlist}
  \item One says that $\mathcal{S}$ is \emph{thick} (or \emph{{\'e}paisse}) if it is closed under direct summands in $\mathcal{T}$. See Neeman \cite[Dfn.~2.1.6]{Nee} (and Rickard \cite[Prop.~1.3]{MR1027750}).
  
  \item One says that $\mathcal{S}$ is \emph{localizing} if it is closed under coproducts in $\mathcal{T}$. Such a subcategory is thick, see B\"{o}kstedt and Neeman \cite[Dfn.~1.3, Rmk.~1.4, and Prop.~3.2]{MR1214458}.
\end{itemlist}
Finally, let $\mathbb{S}$ be any collection (usually a set) of objects in $\mathcal{T}$.
\begin{itemlist}
  \item Write $\operatorname{Thick}(\mathbb{S})$ for the \emph{thick subcategory generated by $\mathbb{S}$}, that is, the smallest thick triangulated subcategory of $\mathcal{T}$ containing $\mathbb{S}$.
  
  \item Write $\operatorname{Loc}(\mathbb{S})$ for the \emph{localizing subcategory generated by $\mathbb{S}$}, that is, the smallest localizing triangulated subcategory of $\mathcal{T}$ containing $\mathbb{S}$.
\end{itemlist}
\end{dfn}

The next two results, \lemref{compact} and \thmref{HoM}, are akin to Hovey's \cite[Thm.~7.4.3 and Cor.~7.4.4]{modcat}. Our proof of \thmref{HoM} uses the Ext-formula from \prpref{Ext-formula}.

\begin{lem}
  \label{lem:compact}
  Let $\mathcal{M}$ be a cocomplete abelian category equipped with a hereditary abelian model structure, and assume that the class $\mathcal{R}$ of fibrant objects is closed under coproducts in $\mathcal{M}$ (for example, if\, $\mathcal{R}=\mathcal{M}$). Let $C$ be an object in $\mathcal{M}$ that is both cofibrant and fibrant. If $C$ is~compact in $\mathcal{M}$, then $C$ is also compact in the homotopy category $\operatorname{Ho}(\mathcal{M})$.
\end{lem}

\begin{proof}
Let $(\mathcal{Q},\mathcal{W},\mathcal{R})$ be the hereditary Hovey triple that corresponds to the given hereditary abelian model structure on $\mathcal{M}$. As the object $C \in \mathcal{M}$ is both cofibrant and fibrant, it can be viewed as an object in the stable category \smash{$\operatorname{Stab}(\mathcal{M}_{\mathrm{cf}}) = (\mathcal{Q} \cap \mathcal{R})\big/(\mathcal{Q}\cap \mathcal{W} \cap \mathcal{R})$}. As this category is equivalent to $\operatorname{Ho}(\mathcal{M})$, it suffices to show that compactness of $C$ in~$\mathcal{M}$ implies compactness of $C$ in $\operatorname{Stab}(\mathcal{M}_{\mathrm{cf}})$. Let $\{X_i\}_{i \in I}$ be any family of objects in $\mathcal{Q} \cap \mathcal{R}$. Assuming that $C$ is compact in $\mathcal{M}$, we know that the canonical homomorphism 
\begin{equation}
  \label{eq:sigma}
    \textstyle
    \sigma \colon \bigoplus_{i \in I}\Hom{\mathcal{M}}{C}{X_i}
    \longrightarrow
    \Hom{\mathcal{M}}{C}{\bigoplus_{i \in I}X_i}
\end{equation}
is bijective. By definition, $\sigma$ is given by $(\alpha_i)_{i \in I} \mapsto \sum_{i \in I}\iota_i\,\alpha_i$ where $\iota_{\!j} \colon X_{\!j} \to \bigoplus_{i \in I}X_i$ is the canonical injection. Recall that in any model category, the class $\mathcal{Q}$ of cofibrant objects is closed under coproducts while the class $\mathcal{R}$ of fibrant objects is closed under products. We have assumed that that $\mathcal{R}$ is closed under coproducts in $\mathcal{M}$, and hence so is the intersection $\mathcal{Q} \cap \mathcal{R}$. This implies that the coproduct of the family $\{X_i\}_{i \in I}$ computed in $\operatorname{Stab}(\mathcal{M}_{\mathrm{cf}})$ is just $\bigoplus_{i \in I}X_i$\,, i.e.~the coproduct in $\mathcal{M}$, and the canonical injections are \mbox{$[\iota_{\!j}] \colon X_{\!j} \to \bigoplus_{i \in I}X_i$}\,, where $[\,\cdot\,] \colon \mathcal{Q} \cap \mathcal{R} \to \operatorname{Stab}(\mathcal{M}_{\mathrm{cf}})$ denotes the canonical functor. Consequently, for the category $\mathcal{A}= \operatorname{Stab}(\mathcal{M}_{\mathrm{cf}})$, the canonical homomorphism from \dfnref{compact} takes the form
  \begin{align*}
    \textstyle
    \tau \colon \bigoplus_{i \in I}\uHom{\mathcal{M}}{C}{X_i}
    &\longrightarrow
    \textstyle
    \uHom{\mathcal{M}}{C}{\bigoplus_{i \in I}X_i}
    \qquad \text{given by}
    \\
    \textstyle
    ([\alpha_i])_{i \in I} 
    &\longmapsto 
    \textstyle
    \sum_{i \in I}[\iota_i][\alpha_i] \,=\, 
    \big[ \sum_{i \in I}\iota_i\,\alpha_i \big]\;.
  \end{align*}
Here we have used the same notation as in the proof of \prpref{Ext-formula}, see \eqref{uHom}. We want to show that $\tau$ is an isomorphism. In view of the definitions of $\tau$ and the hom-sets $\underline{\operatorname{Hom}}_{\mathcal{M}}$, this amounts to showing that the isomorphism $\sigma$ from \eqref{sigma} restricts to an isomorphism
  \begin{equation*}
    \textstyle
    \sigma \colon
    \bigoplus_{i \in I}\mathcal{I}(C,X_i)
    \longrightarrow
    \mathcal{I}(C,\bigoplus_{i \in I}X_i)\;.
  \end{equation*}
This is straightforward to prove and left to the reader.
\end{proof}

\begin{thm}
  \label{thm:HoM}
  Let $\mathcal{M}$ be a cocomplete abelian category equipped with a hereditary abelian model structure induced by the Hovey triple $(\mathcal{Q},\mathcal{W}, \mathcal{R})$. Assume that:
\begin{rqm}
\item The model structure is projective, i.e.~one has $\mathcal{R}=\mathcal{M}$.
\item The cotorsion pair $(\mathcal{Q},\mathcal{W})$ is generated by a set of compact objects in $\mathcal{M}$, i.e.~there exists a subset\, $\mathbb{S} \subseteq \mathcal{M}^\mathrm{c}$ with\, $\mathbb{S}^\perp = \mathcal{W}$.
\end{rqm}  
In this case, the homotopy category of $\mathcal{M}$ is a compactly generated triangulated \mbox{category}; in fact, $\mathbb{S}$ is a set of compact objects in $\operatorname{Ho}(\mathcal{M})$ that generates $\operatorname{Ho}(\mathcal{M})$. In particular, 
\begin{equation*}
  \operatorname{Ho}(\mathcal{M})
  \,=\, 
  \operatorname{Loc}(\mathbb{S})
  \qquad \text{and} \qquad
  \operatorname{Ho}(\mathcal{M})^\mathrm{c}
  \,=\, 
  \operatorname{Thick}(\mathbb{S})\;.
\end{equation*}
\end{thm}

\begin{proof}
  First observe that $\mathbb{S}$ is contained in $\mathcal{Q}$ since $\mathbb{S}^\perp = \mathcal{W}$ implies that $\mathbb{S} \subseteq {}^\perp(\mathbb{S}^\perp) = {}^\perp\mathcal{W} = \mathcal{Q}$. Evidently, $\mathbb{S}$ is also contained in $\mathcal{R}=\mathcal{M}$. We have assumed that $\mathbb{S} \subseteq \mathcal{M}^\mathrm{c}$, so \lemref{compact} implies that $\mathbb{S}$ is a set of compact objects in $\operatorname{Ho}(\mathcal{M})$. To see that $\mathbb{S}$ generates $\operatorname{Ho}(\mathcal{M})$, in the sense of   \dfnref{compactly-generated}, note that for every $S \in \mathbb{S}$ and $Y \in \mathcal{M}$ one has:
  \begin{equation*}
    \Hom{\operatorname{Ho}(\mathcal{M})}{S}{Y} 
    \,\cong\,
    \Hom{\operatorname{Ho}(\mathcal{M})}{S}{\upSigma(\upSigma^{-1}Y)}
    \,\cong\,
    \Ext{\mathcal{M}}{1}{S}{\cofr(\upSigma^{-1}Y)}\;.
  \end{equation*}  
Indeed, the first isomorphism is trivial and the second follows from \prpref{Ext-formula}, where it has been used that every object in $\mathcal{M}$ is a fibrant replacement of itself (since $\mathcal{R}=\mathcal{M}$) and $S$ is a cofibrant replacement of itself (since $\mathbb{S} \subseteq \mathcal{Q}$). Thus, if \smash{$\Hom{\operatorname{Ho}(\mathcal{M})}{S}{Y}=0$} holds for every $S \in \mathbb{S}$, then $\cofr(\upSigma^{-1}Y)$ belongs to $\mathbb{S}^\perp = \mathcal{W}$, and hence one has $\cofr(\upSigma^{-1}Y) =0$ in $\operatorname{Ho}(\mathcal{M})$. But in $\operatorname{Ho}(\mathcal{M})$ there is an isomorphism $\cofr(\upSigma^{-1}Y) \cong \upSigma^{-1}Y$, and consequently $\upSigma^{-1}Y = 0$ in $\operatorname{Ho}(\mathcal{M})$, which implies that $Y = 0$. This proves that $\mathbb{S}$ generates $\operatorname{Ho}(\mathcal{M})$.
  
The last assertion is now an immediate consequence of Neeman \cite[Thm.~2.1]{MR1308405}.
\end{proof}

To parse the next result, recall that the definition of the (co)homology functors from \ref{consequences-of-srp} only requires the Strong Retraction Property (condition \eqref{strong-retraction} in \ref{Bbbk}). Recall also the functors $\Cq{q}$ and $\Kq{q}$ from \ref{recall-C-K}.

\begin{prp}
  \label{prp:H-preserves}
  Assume that $Q$ is Hom-finite, locally bounded and satisfies the Strong Retraction Property, i.e.~conditions \eqref{Homfin}, \eqref{locbd}, and \eqref{strong-retraction} in \ref{Bbbk} hold, and that $\Bbbk$ is noetherian. For every $q \in Q$ and $i \geqslant 0$ the (co)homology functors
  \begin{equation*}
     \cH[i]{q},\, \hH[i]{q} \colon \lMod{Q,\,\alg} \longrightarrow \lMod{\alg}
  \end{equation*}
  preserve products and filtered colimits. In particular, the left adjoint functor $\Cq{q}$ preserves products and right adjoint functor $\Kq{q}$ preserves filtered colimits.
\end{prp}

\begin{proof}
  Under the given assumptions \cite[Lem.~5.6]{HJ2021} shows that the Grothendieck categories $\lMod{Q}$ and $\rMod{Q}$ are locally noetherian; in fact, $\{Q(p,-)\}_{p \in Q}$ and $\{Q(-,p)\}_{p \in Q}$ are generating families of noetherian projective objects in these categories. Thus, the noetherian objects $\stalkco{q} \in \lMod{Q}$ and $\stalkcn{q} \in \rMod{Q}$ from \ref{consequences-of-srp} have projective resolutions,
\begin{equation*}
  \text{$\cpx{P} \,=\, \cdots \to P_1 \to P_0 \to 0$ \ in \ $\lMod{Q}$ \quad \ and \ \quad $\cpx{P}' \,=\, \cdots \to P'_1 \to P'_0 \to 0$ \ in \ $\rMod{Q}$\;,}
\end{equation*}
where each $P_i$ and each $P'_i$ is a finite direct sum of representable objects, that is, objects of the form $Q(p,-)$ and $Q(-,p)$, respectively. By the Yoneda isomorphisms \mbox{\cite[eq. (3.10)]{Kelly}} the functors $\Hom{Q}{Q(p,-)}{?}$ and $\ORt{Q(-,p)}{{?}}$ are both naturally isomorphic to the evaluation functor $\Eq{p}(?) \colon \lMod{Q,\,\alg} \to \lMod{\alg}$ from \ref{recall-F-G}, which evidently preserves all limits and colimits. It follows that for every $i \geqslant 0$ the functors $\Hom{Q}{P_i}{?}$ and $\ORt{P'_i}{{?}}$ preserve all limits and colimits and hence so does the functors $\Hom{Q}{\cpx{P}}{?}$ and $\ORt{\cpx{P}'}{{?}}$, which goes from $\lMod{Q,\,\alg}$ to the category $\operatorname{Ch}\mspace{1mu}\alg$ of chain complexes of left $\alg$-modules. Since  homology $\operatorname{H}_* \colon \operatorname{Ch}\mspace{1mu}\alg \to \lMod{\alg}$ preserves products and filtered colimits, it follows that the functors
  \begin{equation*}
    \cH[i]{q} \,=\, \Ext{Q}{i}{\stalkco{q}}{?}
    \,=\, \operatorname{H}_{-i}\Hom{Q}{\cpx{P}}{?}
    \quad \text{ and } \quad
    \hH[i]{q} \,=\, \Tor{Q}{i}{\stalkcn{q}}{?}
    \,=\, \operatorname{H}_{i}(\ORt{\cpx{P}'}{{?}})
  \end{equation*}
preserve products and filtered colimits too. The last assertion in the proposition now follows immediately as one has $\Cq{q} = \hH[0]{q}$ and $\Kq{q} = \cH[0]{q}$ by definition.
\end{proof}

\begin{rmk}
  \label{rmk:FPinfty}
  Let $\mathcal{M}$ be a Grothendieck category. Following Gillespie \cite[Dfn.~3.1]{MR3594536}, an object $F \in \mathcal{M}$ is said to be of \emph{type $FP_\infty$} if the functors $\Ext{\mathcal{M}}{i}{F}{-}$  preserve filtered colimits for all $i \geqslant 0$.  Evidently, every such object is also compact in the sense of \dfnref{compact}. 
  
As mentioned in the proof above, the category $\lMod{Q}$ is locally noetherian and the object $\stalkco{q}$ is noetherian (under the assumptions in \prpref{H-preserves}). Thus \cite[Thm.~3.17]{MR3594536} shows that $\stalkco{q}$  is of type $FP_\infty$, and hence $\cH[i]{q} = \Ext{Q}{i}{\stalkco{q}}{?}$ preserves filtered colimits for every $i \geqslant 0$. So part of the statement in \prpref{H-preserves} follows directly from \cite{MR3594536}.
\end{rmk}

\begin{cor}
  \label{cor:Sq-compact}
  Adopt the assumptions from \prpref{H-preserves}. If $M \in \lMod{\alg}$ is finitely generated, then~$\Sq{q}(M)$ is a compact object in $\lMod{Q,\,\alg}$ for every $q \in Q$.
\end{cor}

\begin{proof}
  Fix $q \in Q$. By \ref{recall-C-K} there is a natural isomorphism,
  \begin{equation}
    \label{eq:composite-functor}
    \Hom{Q,\,\alg}{\Sq{q}(M)}{-} \,\cong\, \Hom{\alg}{M}{\Kq{q}(-)}\;.
  \end{equation}
The functor $\Kq{q}$ preserves coproducts by \prpref{H-preserves}, and $\Hom{\alg}{M}{-}$ preserves coproducts as $M$ is finitely generated. Thus, the functor in \eqref{composite-functor} preserves coproducts.
\end{proof}

We are now in a position to prove Theorem~D from the Introduction.

\begin{proof}[Proof of Theorem~D]
  We will apply \thmref{HoM} to the abelian category $\mathcal{M} = \lMod{Q,\,\alg}$ equipped with the projective model structure corresponding to the hereditary Hovey triple
\begin{equation*}
  (\mathcal{Q},\mathcal{W}, \mathcal{R}) \,=\, 
  ({}^\perp\lE{},\lE{},\mathcal{M})\;,
\end{equation*}
see \ref{summary-of-HJ2021}. It remains to verify condition \prtlbl{2} in \thmref{HoM}. We have $\mathbb{S} \subseteq \mathcal{M}^\mathrm{c}$ by \corref{Sq-compact}. Furthermore, for every $X \in \mathcal{M}=\lMod{Q,\,\alg}$ and $q \in Q$ there are isomorphisms:
\begin{equation*}
  \Ext{Q,\,\alg}{1}{\Sq{q}(\alg)}{X}
  \,\cong\,
  \Ext{Q,\,\alg}{1}{\stalkco{q} \otimes_\Bbbk \alg}{X}
  \,\cong\,
  \Ext{Q}{1}{\stalkco{q}}{X^\natural}
  \,\cong\,
  \cH{q}(X^\natural)
  \,\cong\,
  \cH{q}(X)^\natural\;.
\end{equation*}  
Indeed, the $1^\mathrm{st}$ isomorphism holds by definition, see   \ref{recall-C-K}. By \cite[Lem.~7.10]{HJ2021}, the stalk functor $\stalkco{q} \colon Q \to \lMod{\Bbbk}$ satisfies $\stalkco{q}(p) \in \lPrj{\Bbbk} \subseteq \lGPrj{\Bbbk}$ for every $p \in Q$. Hence $\stalkco{q}$ is a Gorenstein projective object in $\lMod{Q}$ by \cite[Thm.~2.7]{HJ2021}; in symbols: $\stalkco{q} \in \lGPrj{Q}$. Thus, the 
$2^\mathrm{nd}$ isomorphism above follows from \cite[Lem.~4.3(a)]{HJ2021}. The $3^\mathrm{rd}$ and $4^\mathrm{th}$ isomorphisms hold by \cite[Dfn.~7.11 and Rmk.~7.12]{HJ2021}. From these isomorphisms  the first equivalence below follows; the second equivalence holds by \cite[Thm.~7.1]{HJ2021}:
\begin{equation*}
  \Ext{Q,\,\alg}{1}{\Sq{q}(\alg)}{X}=0 \text{ for each $q \in Q$}
  \ \iff \
  \cH{q}(X)=0 \text{ for each $q \in Q$}  
  \ \iff \
  X \in \lE{}\,.
\end{equation*}
Consequently, one has $\mathbb{S}^\perp = \lE{} =\mathcal{W}$, as desired.
\end{proof}

\section{Perfect objects in $\QSD{Q}{\alg}$}

\label{sec:Perfect}

In this section, we introduce perfect objects and compare them to compact objects.

\begin{dfn}
  \label{dfn:support}
  The \emph{support} of an object $X$ in $\lMod{Q,\,\alg}$ is defined as
  \begin{equation*}
     \supp{X} \,=\,
     \{q \in Q \ | \ X(q) \neq 0 \}\;;
  \end{equation*}
  it is a subset of $\operatorname{Ob}(Q)$. We say that $X$ has \emph{finite support} if $\supp{X}$ is a finite set. 
\end{dfn} 

The next lemma follows immediately from the definition.

\begin{lem}
  \label{lem:ses-support}
  For every short exact sequence $0 \to X' \to X \to X'' \to 0$ in $\lMod{Q,\,\alg}$ one has
  \begin{equation*}
     \supp{X} \,=\, \supp{X'} \,\cup\, \supp{X''}\;.
  \end{equation*}
  Thus, the class of objects in $\lMod{Q,\,\alg}$ with finite support is a Serre subcategory. \qed
\end{lem}

\begin{dfn}
  \label{dfn:perfect}
  We introduce the following notions.
\begin{itemlist}  
\item An object $K$ in $\lMod{Q,\,\alg}$ is called \emph{strictly perfect} if it has finite support and $K(q)$ is a finitely generated projective left $\alg$-module for every $q \in Q$. We set
  \begin{equation*}
    \finsupp{Q,\,\alg} \,=\,
    \big\{
      K \in \lMod{Q,\,\alg}\,\big|\,
      \text{$K$ is strictly perfect}
    \big\}.
  \end{equation*}
  
\item Assume \stpref{nilpotent}. An object $X$ in \smash{$\QSD{Q}{\alg}$} is called \emph{perfect} if it has a strictly perfect cofibrant replacement in the projective model structure on $\lMod{Q,\,\alg}$; that is:
\begin{equation*}
  \qquad
  \text{$X$ is perfect} \ \ \ \iff \ \ \
  \left\{\!\!
  \begin{array}{l}
  \text{There is an isomorphism $X \cong K$ in $\QSD{Q}{\alg}$}
  \\
  \text{for some object $K \in {}^\perp\lE{} \,\cap\, \finsupp{Q,\,\alg}$\;.}
  \end{array}
  \right.
\end{equation*}
We set
  \begin{equation*}
    \QSDperf{Q}{\alg} \,=\,
    \big\{
      X \in \QSD{Q}{\alg}\,\big|\,
      \text{$X$ is perfect}
    \big\}.
  \end{equation*}
\end{itemlist}  
\end{dfn} 

\begin{exa}
  As mentioned in the Introduction: In the special case of chain complexes, strictly perfect objects are bounded complexes of finitely generated modules.
\end{exa}

Beware that a strictly perfect object need not be perfect; see \thmref{counterexample} for a counterexample. On the positive side, one does have Theorem~B (part \prtlbl{c}) from the Introduction, which will be proved later in this section.

The next result follows immediately from \lemref{ses-support} and the definition, \dfnref[]{perfect}, of strictly perfect objects.

\begin{lem}
  \label{lem:extension}
Let $0 \to X' \to X \to X'' \to 0$ be an exact sequence in $\lMod{Q,\,\alg}$ with $X'' \in \finsupp{Q,\,\alg}$. Then one has $X' \in \finsupp{Q,\,\alg}$ if and only if $X \in \finsupp{Q,\,\alg}$. \qed
\end{lem}

\begin{con}
  \label{con:FF-GG}
   Consider the adjoint triple from \ref{recall-F-G}. We introduce the notation:
  \begin{itemlist}
  \item $\varepsilon_q$ for the counit of the adjunction $(\Fq{q},\Eq{q})$.
  \item $\eta_q$ for the unit of the adjunction $(\Eq{q},\Gq{q})$.
  \end{itemlist}
  Consider the endofunctors on the category $\lMod{Q,\,\alg}$ given by
  \begin{equation*}
    \textstyle
    \mathbb{F} \,=\, \bigoplus_{q \mspace{1mu}\in\mspace{1mu} Q}\, \Fq{q}\Eq{q}
    \qquad \text{and} \qquad
    \mathbb{G} \,=\, \prod_{q \mspace{1mu}\in\mspace{1mu} Q}\, \Gq{q}\Eq{q}\;.
  \end{equation*}
  By the universal properties of coproducts and products, there exist unique morphisms $\varepsilon$~and $\eta$ in $\lMod{Q,\,\alg}$ that make the diagrams 
  \begin{equation}
    \label{eq:two-triangles}
    \hspace*{-4.7ex}
    \begin{gathered}
    \xymatrix@C=3pc{
    \Fq{p}\Eq{p}
    \ar[d]_-{\iota_p} \ar[dr]^-{\varepsilon_p}
    & {}
    \\
    \!\!\!\!\!\!\!\!\!\!\mathbb{F}\,=\,\bigoplus_{q \mspace{1mu}\in\mspace{1mu} Q}\, \Fq{q}\Eq{q}
    \ar@{.>}[r]_-{\varepsilon}
    &
    \mathrm{Id}
    }
    \end{gathered}
    \qquad \text{and} \qquad
    \begin{gathered}
    \xymatrix@C=3pc{
    {} & 
    \Gq{p}\Eq{p}
    \ar[d]^-{\pi_p} 
    \\
    \mathrm{Id}
    \ar[ur]^-{\eta_p}
    \ar@{.>}[r]_-{\eta}
    &
    \prod_{q \mspace{1mu}\in\mspace{1mu} Q}\, \Gq{q}\Eq{q} \,=\, \mathbb{G}\!\!\!\!\!\!\!\!\!\!\!\!\!\!\!\!\!\!\!\!\!\!\!\!\!
    }
    \end{gathered}
  \end{equation}    
  commutative for every $p \in Q$; here $\iota_p$ is the injection, $\pi_p$ is the projection, and $\mathrm{Id}$ denotes the identity functor on $\lMod{Q,\,\alg}$. We set
\begin{equation*}
  \mathbb{K} \,=\, \Ker{\varepsilon}
  \qquad \text{and} \qquad
  \mathbb{C} \,=\, \Coker{\eta}\;,
\end{equation*}  
which are also endofunctors on $\lMod{Q,\,\alg}$.
\end{con}

\begin{prp}
  \label{prp:mathfrak-ses}
With the notation from \conref{FF-GG} one has short exact sequen\-ces of endofunctors on $\lMod{Q,\,\alg}$,
\begin{equation*}
  \mathfrak{F} \,=\,
  \xymatrix@C=1.1pc{
    0 \ar[r] & \mathbb{K} \ar[r] & \mathbb{F} \ar[r]^-{\varepsilon} & \mathrm{Id} \ar[r] & 0
  }
  \qquad \text{and} \qquad
  \mathfrak{G} \,=\,
  \xymatrix@C=1.1pc{
    0 \ar[r] & \mathrm{Id} \ar[r]^-{\eta} & \mathbb{G} \ar[r] & \mathbb{C} \ar[r] & 0
  }.
\end{equation*}  
For every object $X$ in $\lMod{Q,\,\alg}$, the short exact sequences $\mathfrak{F}(X)$ and $\mathfrak{G}(X)$ in $\lMod{Q,\,\alg}$ are objectwise split (see \dfnref{objectwise-split}).
\end{prp}

\begin{proof}
  We only prove the assertions about the sequence $\mathfrak{F}$ as the sequence $\mathfrak{G}$ can be hand\-led similarly. For every $X$ in $\lMod{Q,\,\alg}$, \conref{FF-GG} yields an exact sequence
\begin{equation*}
  \xymatrix@C=1.5pc{
    0 \ar[r] & \mathbb{K}(X) \ar[r] & \mathbb{F}(X) \ar[r]^-{\varepsilon^X} & X
  }.
\end{equation*}
Here we have used the notation $\varepsilon^X$ instead of $\varepsilon(X)$. We must prove that $\varepsilon^X$ is epic, i.e.~that the morphism $\varepsilon^X(p)$ in $\lMod{\alg}$ is epic for every $p \in Q$. We will even show that $\varepsilon^X(p)$ is  split epic, which also takes care of the last assertion in the lemma. By evaluating the leftmost diagram in \eqref{two-triangles} at $X$ and applying the functor $\Eq{p}$ to the resulting diagram, one gets
\begin{equation}
  \label{eq:EEE}
  \Eq{p}(\varepsilon^X) \circ \Eq{p}(\iota_p^X) \,=\, \Eq{p}(\varepsilon_p^X)\;.
\end{equation}
Let $\zeta_p$ be the unit of the adjunction $(\Fq{p},\Eq{p})$. By the triangle (or the zigzag) identities for adjoint functors, see \cite[Chap.~IV.1, Thm.~1(ii)]{mac}, the composite
\begin{equation*}
  \xymatrix@C=2.5pc{
    \Eq{p} \ar[r]^-{\zeta_p \Eq{p}} & 
    \Eq{p}\Fq{p}\Eq{p} \ar[r]^-{\Eq{p} \varepsilon_p} & \Eq{p}
  }
\end{equation*}
is the identity transformation on $\Eq{p}$, so the morphism $\zeta_p^{\,\Eq{p}(X)}$ is a right inverse of 
\smash{$\Eq{p}(\varepsilon_p^X)$}. Combined with the identity \eqref{EEE}, it follows that the composite \smash{$\Eq{p}(\iota_p^X) \circ \zeta_p^{\,\Eq{p}(X)}$} is a right inverse of the morphism \smash{$\Eq{p}(\varepsilon^X) = \varepsilon^X(p)$}, as desired.
\end{proof}

\begin{lem}
  \label{lem:ses-K}
 Assume that $Q$ is Hom-finite and locally bounded, i.e.~$Q$ satisfies  conditions \eqref{Homfin} and \eqref{locbd} in \ref{Bbbk}. For every object $X \in \lMod{Q,\,\alg}$ the following assertions hold.
\begin{prt}
\item If $X$ has finite support, then $\mathbb{F}(X)$, $\mathbb{G}(X)$, $\mathbb{K}(X)$, and $\mathbb{C}(X)$ have finite support.
\item If $X$ belongs to $\finsupp{Q,\,\alg}$, then $\mathbb{F}(X)$, $\mathbb{G}(X)$, $\mathbb{K}(X)$, and $\mathbb{C}(X)$ belong to $\finsupp{Q,\,\alg}$.

\item If $X$ belongs to $\finsupp{Q,\,\alg}$, then $\mathbb{F}(X)$ is a finitely presented projective object in $\lMod{Q,\,\alg}$.
\end{prt}  
\end{lem}

\begin{proof}
  Recall from \prpref{mathfrak-ses} that there are short exact sequences,
\begin{equation*}
  \mathfrak{F}(X) \,=\,\,
    0 \to \mathbb{K}(X) \to \mathbb{F}(X) \to X \to 0
   \quad \text{ and } \quad
   \mathfrak{G}(X) \,=\,\,
    0 \to X \to \mathbb{G}(X) \to \mathbb{C}(X) \to 0\,.
\end{equation*}  
  \proofoftag{a} Assume that $X$ has finite support. By the short exact sequences above and by \lemref{ses-support} it suffices to argue that the objects $\mathbb{F}(X)$ and $\mathbb{G}(X)$ have finite support. By \conref{FF-GG} and by the definition of support, see \dfnref[]{support}, there are for each $p \in Q$ equalities:
\begin{equation} 
  \label{eq:F-G-1} 
  \begin{aligned}
    \textstyle
    \mathbb{F}(X)(p) 
    &\,=\, 
    \textstyle
    \bigoplus_{q \mspace{1mu}\in\mspace{1mu} Q}\, \Fq{q}(X(q))(p)
    \,=\, 
    \bigoplus_{q \mspace{1mu}\in\mspace{1mu} \supp{X}}\, \Fq{q}(X(q))(p)\;,
    \\
    \textstyle
    \mathbb{G}(X)(p) 
    &\,=\, 
    \textstyle
    \prod_{q \mspace{1mu}\in\mspace{1mu} Q}\, \Gq{q}(X(q))(p)
    \,=\, 
    \prod_{q \mspace{1mu}\in\mspace{1mu} \supp{X}}\, \Gq{q}(X(q))(p)\;.
  \end{aligned}  
\end{equation}  
  Now fix $q \in \supp{X}$ and recall from \ref{recall-F-G} that one has
\begin{equation}
  \label{eq:F-G-2} 
  \begin{aligned}
  \Fq{q}(X(q))(p) &\,=\, Q(q,p) \otimes_\Bbbk X(q)\;,
  \\
  \Gq{q}(X(q))(p) &\,=\, \Hom{\Bbbk}{Q(p,q)}{X(q)}\;.
  \end{aligned}
\end{equation}
We see that if $\Fq{q}(X(q))(p) \neq 0$ then $p$ must be in the set $\operatorname{N}_+(q)$, and if $\Gq{q}(X(q))(p) \neq 0$ then $p$ must be in the set $\operatorname{N}_-(q)$; the sets $\operatorname{N}_\pm(q)$ are defined in condition \eqref{locbd} in \ref{Bbbk}. Hence there are inclusions:
\begin{equation*}
  \textstyle
  \supp{\mathbb{F}(X)} \,\subseteq \ \bigcup_{q \mspace{1mu}\in\mspace{1mu} \supp{X}} \,\operatorname{N}_+(q)
    \quad \text{ and } \quad
  \supp{\mathbb{G}(X)} \,\subseteq \ \bigcup_{q \mspace{1mu}\in\mspace{1mu} \supp{X}} \,\operatorname{N}_-(q)\;.
\end{equation*}
By assumption, the set $\supp{X}$ is finite, and so are the sets $\operatorname{N}_{\pm}(q)$ for each $q$ as $Q$ is locally bounded. Thus, the inclusions above show that $\mathbb{F}(X)$ and $\mathbb{G}(X)$ have finite support.

\proofoftag{b} Assume that $X$ belongs to $\finsupp{Q,\,\alg}$. From part \prtlbl{a} we know that $\mathbb{F}(X)$, $\mathbb{G}(X)$, $\mathbb{K}(X)$, and $\mathbb{C}(X)$ have finite support, so it remains to show that for each $p \in Q$ the left $\alg$-modules $\mathbb{F}(X)(p)$, $\mathbb{G}(X)(p)$, $\mathbb{K}(X)(p)$, and $\mathbb{C}(X)(p)$ are finitely generated projective. \prpref{mathfrak-ses} shows that for each $p \in Q$ the sequences $\mathfrak{F}(X)(p)$ and $\mathfrak{G}(X)(p)$ are split exact in $\lMod{\alg}$, so it suffices to see that $\mathbb{F}(X)(p)$ and $\mathbb{G}(X)(p)$ are finitely generated projective. Fix $p \in Q$. As $Q$ is Hom-finite, the $\Bbbk$-modules $Q(q,p)$ and $Q(p,q)$ are finitely generated projective for every $q \in Q$. As $X(q)$ is a finitely generated projective left $\alg$-module for every $q \in Q$, it follows from \eqref{F-G-2} that $\Fq{q}(X(q))(p)$ and $\Gq{q}(X(q))(p)$ are finitely generated projective left $\alg$-modules for every $q \in Q$. Since $\supp{X}$ is a finite set, \eqref{F-G-1} now shows that $\mathbb{F}(X)(p)$ and $\mathbb{G}(X)(p)$ are finite direct sums of finitely generated projective left $\alg$-modules; hence $\mathbb{F}(X)(p)$ and $\mathbb{G}(X)(p)$ are finitely generated projective left $\alg$-modules.

\proofoftag{c} From \eqref{F-G-1} we know that 
$\mathbb{F}(X)$ is the finite direct sum $\bigoplus_{q \mspace{1mu}\in\mspace{1mu} \supp{X}}\, \Fq{q}(X(q))$; and by \cite[Lem.~3.11]{HJ2021} each object $\Fq{q}(X(q))$ is finitely presented and projective in $\lMod{Q,\,\alg}$.
\end{proof}

\begin{prp}
  \label{prp:strictly-perfect-is-compact}
 Assume that $Q$ is Hom-finite and locally bounded, i.e.~$Q$ satisfies~conditions \eqref{Homfin} and \eqref{locbd} in \ref{Bbbk}.
  Every strictly perfect object in $\lMod{Q,\,\alg}$ is of type $FP_\infty$ (see \rmkref{FPinfty}); in particular, there is an inclusion:
\begin{equation*}
  \finsupp{Q,\,\alg} \,\subseteq\, (\lMod{Q,\,\alg})^\mathrm{c}\;.
\end{equation*}  
\end{prp}

\begin{proof}
  Consider the endofunctors $\mathbb{F}$ and $\mathbb{K}$ on \smash{$\lMod{Q,\,\alg}$} from \conref{FF-GG}. For $i \geqslant 0$ we denote by $\mathbb{K}^i$ the $i$-fold composite $\mathbb{K} \circ \cdots \circ \mathbb{K}$ with the convention that $\mathbb{K}^0 = \mathrm{Id}$ is the identity functor. Let $X \in \lMod{Q,\,\alg}$.   Pasting together the short exact sequences,
\begin{equation*}
  \mathfrak{F}(\mathbb{K}^i(X)) \,=\, \
    0 \longrightarrow \mathbb{K}^{i+1}(X) \longrightarrow \mathbb{F}(\mathbb{K}^i(X)) \longrightarrow \mathbb{K}^i(X) \longrightarrow 0
   \qquad (i \geqslant 0)
\end{equation*}  
coming from \prpref{mathfrak-ses}, we get a long exact sequence,
\begin{equation}
  \label{eq:fg-proj-res}
  \cdots \longrightarrow \mathbb{F}(\mathbb{K}^3(X)) \longrightarrow \mathbb{F}(\mathbb{K}^2(X)) \longrightarrow \mathbb{F}(\mathbb{K}(X)) \longrightarrow X \longrightarrow 0\;.
\end{equation}  
If $X \in \finsupp{Q,\,\alg}$, then successive applications of \lemref{ses-K}(b) show that $\mathbb{K}^i(X) \in \finsupp{Q,\,\alg}$ for every $i \geqslant 0$, and hence each $\mathbb{F}(\mathbb{K}^i(X))$ is finitely presented and projective by \lemref{ses-K}\prtlbl{c}. Thus \eqref{fg-proj-res} is a projective resolution of $X$ by finitely presented (and projective) objects, and it follows that $X$ is of type $FP_\infty$.
\end{proof}

\begin{cor}
  \label{cor:perfect-is-compact-in-D}
  Assume \stpref{nilpotent}. Every perfect object in $\QSD{Q}{\alg}$ is compact, that is:
\begin{equation*}
  \QSDperf{Q}{\alg} \,\subseteq\, \QSD{Q}{\alg}^\mathrm{c}\;.
\end{equation*}  
\end{cor}

\begin{proof}
The class of compact objects in $\QSD{Q}{\alg}$ is closed under isomorphisms, so to prove the inclusion \smash{$\QSDperf{Q}{\alg} \subseteq \QSD{Q}{\alg}^\mathrm{c}$} it suffices by \dfnref{perfect} to argue that every object \smash{$K \in {}^\perp\lE{} \,\cap\, \finsupp{Q,\,\alg}$} is compact in \smash{$\QSD{Q}{\alg}$}. Recall from \ref{summary-of-HJ2021} that $\QSD{Q}{\alg}$ is the homotopy category of $\lMod{Q,\,\alg}$ equipped with the hereditary abelian model structure where $\mathcal{Q}={}^\perp\lE{}$ is the class of cofibrant objects and every object is fibrant. Thus every object \smash{$K \in {}^\perp\lE{} \,\cap\, \finsupp{Q,\,\alg}$} is both cofibrant and fibrant, and by \prpref{strictly-perfect-is-compact} it is also compact in $\lMod{Q,\,\alg}$. Thus, \lemref{compact} implies that $K$ is also compact in \smash{$\operatorname{Ho}(\lMod{Q,\,\alg}) = \QSD{Q}{\alg}$}.
\end{proof}

Our next goal is to show that $\QSDperf{Q}{\alg}$ is even a triangulated subcategory of $\QSD{Q}{\alg}^\mathrm{c}$.

\begin{prp}
  \label{prp:conflation}
  Assume \stpref{nilpotent}. Consider for $X \in \lMod{Q,\,\alg}$   the exact sequences,
\begin{equation*}
  \mathfrak{F}(X) \,=\,\,
    0 \to \mathbb{K}(X) \to \mathbb{F}(X) \to X \to 0
   \quad \text{ and } \quad
   \mathfrak{G}(X) \,=\,\,
    0 \to X \to \mathbb{G}(X) \to \mathbb{C}(X) \to 0
\end{equation*}   
in $\lMod{Q,\,\alg}$ from \prpref{mathfrak-ses}. The following assertions hold:
\begin{prt}
\item If $X \in {}^\perp\lE{}$, then one has $\mathbb{F}(X), \mathbb{G}(X) \in \lPrj{Q,\,\alg}$ and $\mathbb{K}(X), \mathbb{C}(X) \in {}^\perp\lE{}$, and hence:
\begin{itemlist}
\item The sequences $\mathfrak{F}(X)$ and $\mathfrak{G}(X)$ are conflations in the exact category ${}^\perp\lE{}$.
\item In the triangulated stable category ${}^\perp\lE{}\big/\lPrj{Q,\,\alg}$, whose translation functor we denote by $\upSigma$, there are isomorphisms $\upSigma X \cong \mathbb{C}(X)$ and $\upSigma^{-1}X \cong \mathbb{K}(X)$.
\end{itemlist}

\item If $X \in \lE{}^\perp$, then one has $\mathbb{F}(X), \mathbb{G}(X) \in \lInj{Q,\,\alg}$ and $\mathbb{K}(X), \mathbb{C}(X) \in \lE{}^\perp$, and hence:
\begin{itemlist}
\item The sequences $\mathfrak{F}(X)$ and $\mathfrak{G}(X)$ are conflations in the exact category $\lE{}^\perp$.
\item In the triangulated stable category $\lE{}^\perp\big/\lInj{Q,\,\alg}$, whose translation functor we denote by $\upSigma$, there are isomorphisms $\upSigma X \cong \mathbb{C}(X)$ and $\upSigma^{-1}X \cong \mathbb{K}(X)$.
\end{itemlist}
\end{prt}
\end{prp}

\begin{proof}
  \proofoftag{a} Assume that $X$ belongs to ${}^\perp\lE{}$. For every $q \in Q$ the left $\alg$-module $\Eq{q}(X) = X(q)$ is projective by  \prpref{E-perp-inclusions}\prtlbl{a}, and hence $\Fq{q}\Eq{q}(X)$ is a projective object in $\lMod{Q,\,\alg}$ by \cite[Lem.~3.11]{HJ2021}. Since the class $\lPrj{Q,\,\alg}$ is closed under coproducts, it follows that the object $\mathbb{F}(X) = \bigoplus_{q \in Q} \Fq{q}\Eq{q}(X)$ belongs to $\lPrj{Q,\,\alg}$. \lemref{F-G-iso}
 and \prpref{F-G-locally-finite} now imply that also $\mathbb{G}(X) = \prod_{q \in Q} \Gq{q}\Eq{q}(X)$ belongs to $\lPrj{Q,\,\alg}$. In particular, the objects $X$, $\mathbb{F}(X)$, and $\mathbb{G}(X)$ are in \smash{${}^\perp\lE{}$}. As the exact sequences $\mathfrak{F}(X)$ and $\mathfrak{G}(X)$ are objectwise split by \prpref{mathfrak-ses}, it now follows from \thmref{closure-properties}\prtlbl{a,b} that $\mathbb{K}(X)$ and $\mathbb{C}(X)$ are in ${}^\perp\lE{}$.
  
  As the class ${}^\perp\lE{}$ is closed under extensions in the abelian category $\lMod{Q,\,\alg}$, it naturally inherits the structure of an exact category where the conflations are short exact sequences $0 \to Y' \to Y \to Y'' \to 0$ in $\lMod{Q,\,\alg}$ with $Y',Y, Y'' \in {}^\perp\lE{}$. This explains the first bullet point. The second bullet point is evident from the general construction/definition of the translation functor in the stable category of a Frobenius category; see Happel \cite[Chap.~I\S2]{Happel}. 
    
  \proofoftag{b} Dual to the proof of part \prtlbl{a}.
\end{proof}

\begin{thm}
  \label{thm:perf-triangulated}
  Assume \stpref{nilpotent}. The subcategory \smash{$\QSDperf{Q}{\alg}$} of\, \smash{$\QSD{Q}{\alg}$} is triangulated.
\end{thm}

\begin{proof}
  Write \smash{$\mathcal{T} = {}^\perp\lE{}\big/\lPrj{Q,\,\alg}$} for the stable category of the Frobenius category ${}^\perp\lE{}$. Set
\begin{equation*}
  \mathcal{K}' \,=\, {}^\perp\lE{} \,\cap\, \finsupp{Q,\,\alg}\;.
\end{equation*}
By \cite[Thm.~6.5]{HJ2021} there is an equivalence of triangulated categories, $\QSD{Q}{\alg} \simeq \mathcal{T}$; and via this equivalence, the full subcategory \smash{$\QSDperf{Q}{\alg} \subseteq \QSD{Q}{\alg}$} corresponds, in view of \dfnref{perfect}, to the full subcategory $\mathcal{T}^{\,\mathrm{perf}} \subseteq \mathcal{T}$ defined by
  \begin{equation*}
    \mathcal{T}^{\,\mathrm{perf}} \,=\,
    \left\{
      X \in \mathcal{T}\,
      \left|\!\!
      \begin{array}{l}
        \text{There is an isomorphism $X \cong K$} \\
        \text{in $\mathcal{T}$ for some object $K \in \mathcal{K'}$}
      \end{array}
      \right.\!\!\!\!
    \right\}.
  \end{equation*}  
Thus, it suffices to argue that $\mathcal{T}^{\,\mathrm{perf}}$ is a triangulated subcategory of $\mathcal{T}$. We will verify the requirements in Neeman \cite[Dfn.~1.5.1]{Nee}.

As $\mathcal{K'}$ is clearly an additive subcategory of ${}^\perp\lE{} \subseteq \lMod{Q,\,\alg}$, it follows that $\mathcal{T}^{\,\mathrm{perf}}$ is an additive subcategory of $\mathcal{T}$. By definition, $\mathcal{T}^{\,\mathrm{perf}}$ is closed under isomorphisms in $\mathcal{T}$.

Let $\upSigma$ denote the translation functor in the triangulated category $\mathcal{T}$. We must show that one has $\upSigma(\mathcal{T}^{\,\mathrm{perf}}) = \mathcal{T}^{\,\mathrm{perf}}$; equivalently, that the following two inclusions hold:
\begin{equation*}
  \upSigma(\mathcal{T}^{\,\mathrm{perf}}) \,\subseteq\, 
  \mathcal{T}^{\,\mathrm{perf}}
  \qquad \text{and} \qquad 
  \upSigma^{-1}(\mathcal{T}^{\,\mathrm{perf}}) \,\subseteq\,
  \mathcal{T}^{\,\mathrm{perf}}\;.
\end{equation*}  
By the definiton of $\mathcal{T}^{\,\mathrm{perf}}$, it suffices to prove the inclusions:
\begin{equation*}
  \upSigma(\mathcal{K'}) \,\subseteq\, 
  \mathcal{T}^{\,\mathrm{perf}}
  \qquad \text{and} \qquad 
  \upSigma^{-1}(\mathcal{K'}) \,\subseteq\,
  \mathcal{T}^{\,\mathrm{perf}}\;.
\end{equation*}  
To establish these inclusions, let $X$ be an object in $\mathcal{K'}$. By part \prtlbl{a} in \prpref{conflation} we have 
$\upSigma X \cong \mathbb{C}(X)$ and $\upSigma^{-1}X \cong \mathbb{K}(X)$, so it suffices to argue that the objects $\mathbb{C}(X)$ and $\mathbb{K}(X)$ belong to $\mathcal{K'}$, but this follows from 
\lemref{ses-K}\prtlbl{b} and \prpref{conflation}\prtlbl{a}.

It remains to prove that if $X \to Y \to Z \to \upSigma X$ is a distinguished triangle in $\mathcal{T}$ with $X,Y \in \mathcal{T}^{\,\mathrm{perf}}$, then one has $Z \in \mathcal{T}^{\,\mathrm{perf}}$. Without loss of generality, we can assume that $X$ and $Y$ belong to $\mathcal{K'}$. Let $[\alpha] \colon X \to Y$ be the morphism from $X$ to $Y$ in the given distinguished triangle; here $\alpha \colon X \to Y$ is a morphism in $\lMod{Q,\,\alg}$ (between objects $X,Y \in \mathcal{K'} \subseteq {}^\perp\lE{}$) and $[\alpha]$ denote its class (image) in the stable category $\mathcal{T}$. Consider in the abelian category $\lMod{Q,\,\alg}$ the pushout diagram below where the upper exact sequence comes from \prpref{mathfrak-ses}:
\begin{equation}
  \label{eq:pushout}
  \begin{gathered}
  \xymatrix{
    0 \ar[r] & X \ar[d]_-{\alpha} \ar@{}[dr]|-{\mathrm{(pushout)}} \ar[r] & \mathbb{G}(X) \ar[d] \ar[r] & \mathbb{C}(X) \ar@{=}[d] \ar[r] & 0
    \\
    0 \ar[r] & Y \ar[r] & P \ar[r] & \mathbb{C}(X) \ar[r] & 0\;.\mspace{-8mu}
  }
  \end{gathered}
\end{equation}
Recall from \prpref{conflation}\prtlbl{a} that $\mathbb{C}(X) \cong \upSigma X$ in $\mathcal{T}$. The diagram \smash{$X \xrightarrow{[\alpha]} Y \to P \to \upSigma X$} is a distinguished triangle in $\mathcal{T}$ (called a \emph{standard} triangle), see Happel \cite[Chap.~I\S2, (2.5)]{Happel}, and hence there exists by \cite[Rmk.~1.1.21]{Nee} an isomorphism $Z \cong P$ in $\mathcal{T}$. Thus, to show $Z \in \mathcal{T}^{\,\mathrm{perf}}$ it suffices to argue that $P \in \mathcal{K'}$. As $X \in \mathcal{K'}$ we have, as above, $\mathbb{C}(X) \in \mathcal{K'}$. Since also $Y \in \mathcal{K'}$ we can now apply \lemref{extension} to the bottom exact sequence in \eqref{pushout} to conclude that $P \in \mathcal{K'}$, as desired. 
\end{proof}

We are now ready to show Theorem~C from the Introduction.

\begin{proof}[Proof of Theorem~C]
  \thmref{perf-triangulated} shows that \smash{$\QSDperf{Q}{\alg}$} is a triangulated subcategory of \smash{$\QSD{Q}{\alg}$}. And it is a general fact that the subcategory \smash{$\QSD{Q}{\alg}^\mathrm{c}$} of compact objects is triangulated and thick, see \cite[Property 6.3]{MR1214458} or more generally \cite[Lem.~4.2.4, Rmk.~4.2.6, and Dfn.~4.2.7]{Nee}. \corref{perfect-is-compact-in-D} shows that there is always an inclusion \smash{$\QSDperf{Q}{\alg} \subseteq \QSD{Q}{\alg}^\mathrm{c}$}.
  
For each $q \in Q$ the object $\Sq{q}(\alg)$ is evidently strictly perfect, see \dfnref{perfect}, and it also belongs to ${}^\perp\lE{}$ by \prpref{E-perp-inclusions}\prtlbl{a}. Thus, $\Sq{q}(\alg)$ belongs to \smash{$\QSDperf{Q}{\alg}$}, again by \dfnref{perfect}. By Theorem~E in the Introduction (proved in \secref{Fibrant}) we have 
\begin{equation*}
  \QSD{Q}{\alg}^\mathrm{c}
  \,=\, 
  \operatorname{Thick}(\{\Sq{q}(\alg) \,|\, q \in Q \})\;.
\end{equation*}  
It follows that if the subcategory \smash{$\QSDperf{Q}{\alg}$} is thick, then one has \smash{$\QSD{Q}{\alg}^\mathrm{c} \subseteq \QSDperf{Q}{\alg}$}, and hence equality holds, i.e.~\smash{$\QSD{Q}{\alg}^\mathrm{c} = \QSDperf{Q}{\alg}$}. Conversely, if this this equality holds, then the subcategory \smash{$\QSDperf{Q}{\alg}$} must be thick, as this is true for \smash{$\QSD{Q}{\alg}^\mathrm{c}$}. 
\end{proof}

The Strong Retraction Property allows one to define a notion of cycles. 

\begin{dfn}
  \label{dfn:cycle}
Assume that $Q$ satisfies the Strong Retraction Property, i.e.~condition \eqref{strong-retraction} in \ref{Bbbk} holds.
A \emph{cycle} in $Q$ (more precisely, a \emph{$\mathfrak{r}$-cycle}, where $\mathfrak{r}$ is the pseudo-radical of $Q$) is a sequence of $n \geqslant 1$ morphisms in $Q$,
\begin{equation*}
  \xymatrix{  
    q_1 \ar[r]^-{g_1} &
    q_2 \ar[r]^-{g_2} &
    \ \cdots \ \ar[r]^-{g_{n-1}} &
    q_n \ar[r]^-{g_n} &
    q_{n+1}
  },
\end{equation*}
where $q_1 = q_{n+1}$ and $0 \neq g_i \in \mathfrak{r}$ for every $i=1,\ldots,n$. The number $n$ is called the \emph{length} of the cycle. A cycle of length $n=1$ is called a \emph{loop}.
\end{dfn}

Recall from \ref{recall-C-K} the functor $\Kq{q}$. Let $X \in \lMod{Q,\,\alg}$. As shown in \cite[Prop.~7.18]{HJ2021} the left $\alg$-module $\Kq{q}(X)$ is isomorphic to $\bigcap_{g \mspace{1mu}\in\mspace{1mu} \mathfrak{r}(q,*)} \Ker{X(g)}$ where the intersection is taken over all morphisms $g \in \mathfrak{r}$ with domain $q$. In particular, $\Kq{q}(X)$ can be seen as a submodule~of~$X(q)$.

\begin{lem}
  \label{lem:Kq-Xq}
Assume that $Q$ satisfies the Strong Retraction Property, i.e.~condition \eqref{strong-retraction} in \ref{Bbbk} holds. Assume furthermore that $Q$ has no cycles. If an object $X \neq 0$ in $\lMod{Q,\,\alg}$ has finite support, then there exists  $q \in \supp{X}$ with $\Kq{q}(X)=X(q)$.
\end{lem}

\begin{proof}
  Suppose (for contradiction) that $\Kq{q}(X) \neq X(q)$ holds for every $q \in \supp{X}$. Since $X$ is non-zero, it has non-empty support, so we can choose an element $q_1 \in \supp{X}$. By assumption we have $\Kq{q_1}(X) \neq X(q_1)$; this means that there exists some $g_1 \colon q_1 \to q_2$ in $\mathfrak{r}$ with $X(g_1) \neq 0$; in particular, $g_1 \neq 0$ and $q_2 \in \supp{X}$. Note that $q_1 \neq q_2$ as $Q$ has no cycles. By assumption we have $\Kq{q_2}(X) \neq X(q_2)$, so as before there exists some $g_2 \colon q_2 \to q_3$ in $\mathfrak{r}$ with $X(g_2) \neq 0$; in particular, $g_2 \neq 0$ and $q_3 \in \supp{X}$. Note that $q_3 \neq q_1$ and $q_3 \neq q_2$ as $Q$ has no cycles. Continuing in this manner, we construct an infinite sequence, 
\begin{equation*}
  \xymatrix{  
    q_1 \ar[r]^-{g_1} &
    q_2 \ar[r]^-{g_2} &
    q_3 \ar[r]^-{g_3} &
    \ \cdots 
  },
\end{equation*}
of morphisms in $\mathfrak{r}$ where each $X(g_i)$ is non-zero and 
$q_1,q_2,q_3,\ldots$ are different objects in $\supp{X}$. However, this contradicts the assumption that $X$ has finite support.
\end{proof}

\begin{rmk}
  An similar argument shows that if an object $X \neq 0$ in $\lMod{Q,\,\alg}$ has finite support, then there also exists some $p \in \supp{X}$ with $\Cq{p}(X)=X(p)$.
\end{rmk}

We now prove Theorem~B from the Introduction.

\begin{proof}[Proof of Theorem~B]
  \proofoftag{b} If $\alg$ has finite left global dimension, the inclusion $\finsupp{Q,\,\alg} \,\subseteq\, {}^\perp\lE{}$ follows from \dfnref{perfect} and Theorem~E from the Introduction (proved in \secref{Fibrant}).
  
Next consider the case where $Q$ has no cycles (and $\alg$ is arbitrary). Let \smash{$X \in \finsupp{Q,\,\alg}$}; we must show that $X$ belongs to ${}^\perp\lE{}$. We argue by induction on the cardinality $n = |\supp{X}|$ of the support of $X$. If $n=0$, then $\supp{X}=\varnothing$ and hence $X=0 \in {}^\perp\lE{}$. Now assume that $n>0$ and that every \smash{$Y \in \finsupp{Q,\,\alg}$} with $|\supp{Y}|<n$ belongs to ${}^\perp\lE{}$. By \lemref{Kq-Xq} there exists $q \in \supp{X}$ with $\Kq{q}(X)=X(q)$. Let $\varepsilon_q$ be the counit of the adjunction $(\Sq{q},\Kq{q})$ from \ref{recall-C-K} and note that the morphism   \vspace*{-1ex}
\begin{equation*}
  \xymatrix{
    \Sq{q}(X(q)) \,=\, \Sq{q}\Kq{q}(X) \ar[r]^-{\varepsilon_q^X} & X
  }
\end{equation*}
is monic as $\varepsilon_q^X(q) \colon X(q) \to X(q)$ is the identity map and $\varepsilon_q^X(p) \colon 0 \to X(p)$ is the zero map for $p \neq q$. Thus there is a short exact sequence in $\lMod{Q,\,\alg}$,
\begin{equation*}
  0 \longrightarrow \Sq{q}(X(q)) \longrightarrow X \longrightarrow Y \longrightarrow 0\;,
\end{equation*}
where $Y(q)=0$ and $Y(p)=X(p)$ for $p \neq q$. It follows that \smash{$Y \in \finsupp{Q,\,\alg}$} with $|\supp{Y}|=n-1$, and hence $Y$ belongs to ${}^\perp\lE{}$ by the induction hypothesis. As the left $\alg$-module $X(q)$ is (finitely generated and) projective, one has $\Sq{q}(X(q)) \in {}^\perp\lE{}$ by \prpref{E-perp-inclusions}\prtlbl{a}. As the class ${}^\perp\lE{}$ is closed under extensions, see \thmref{closure-properties}\prtlbl{a}, we conclude that $X$ is in ${}^\perp\lE{}$.

\proofoftag{c} Having established the inclusion $\finsupp{Q,\,\alg} \subseteq {}^\perp\lE{}$, the assertion follows directly from the definition, \dfnref[]{perfect}, of perfect objects. 

\proofoftag{a} By part \prtlbl{c}, every strictly perfect object is perfect in $\QSD{Q}{\alg}$ and every such object is compact in $\QSD{Q}{\alg}$ by \corref{perfect-is-compact-in-D}.
\end{proof}

We end this paper with an appendix that contains a proof of Theorem~A from the Introduction. In fact, we prove the more general   \thmref{counterexample}, which shows that the conclusions in parts \prtlbl{a}--\prtlbl{c} of Theorem~B (from the Introduction) may fail without further assumptions on the category $Q$ or the ring $\alg$ (like conditions \rqmlbl{1} or \rqmlbl{2} of that theorem). 

\appendix

\section{An example of a strictly perfect non-compact object}

\label{app:example}

Let $\Bbbk$ be any commutative, noetherian, and hereditary ring  (e.g.~$\Bbbk$ is a field or $\Bbbk=\mathbb{Z}$) as in \stpref{nilpotent}, and consider the commutative noetherian $\Bbbk$-algebra 
\begin{equation*}
  \alg \,=\, \Bbbk[X,Y]/(X^2,XY)\;.
\end{equation*}  
We write $x$ and $y$ for the images of $X$ and $Y$ in the quotient ring $A$. Note that in $\alg$ one has $x^2=xy=0$; moreover, $\alg$ is a free $\Bbbk$-module with basis $\{1,x,y,y^2,y^3,\ldots\}$.

We consider the Jordan quiver; it has a single vertex ($*$) and a single loop ($\varepsilon$):
\begin{equation*}
  \upGamma : \quad
  \xymatrix@C=-1.2ex{
  \ast & \ar@(ur,dr)[]^-{\varepsilon}
  }
\end{equation*}
equipped with the relation $\varepsilon^2=0$. Let $Q$ be the path category of $\upGamma$ over $\Bbbk$ modulo the~ideal generated by $\varepsilon^2$; in symbols: $Q=\Bbbk\upGamma/(\varepsilon^2)$. Clearly there is an equivalence of categories,
\begin{equation*}
  \lMod{Q,\,\alg} \,\simeq\, \operatorname{Diff}(\alg)\;,
\end{equation*}
where $\operatorname{Diff}(\alg)$ denotes the category of differential $\alg$-modules. An object in $\operatorname{Diff}(\alg)$ is a pair $(X,d)$ where $X$ is an $\alg$-module and $d$ is an endomorphism of $X$ with $d^2=0$.

\begin{rmk}
  As the path algebra of $\upGamma$ over $\Bbbk$ modulo the ideal generated by $\varepsilon^2$ is the ring $\Bbbk[\varepsilon]/(\varepsilon^2)$ of dual numbers over $\Bbbk$, the category $\lMod{Q,\,\alg}$ is also equivalent to the category of modules over $A[\varepsilon]/(\varepsilon^2)$. However, we will work with the category $\operatorname{Diff}(\alg)$ instead.
\end{rmk}

\begin{ipg}
  \label{Diff}
It is easily seen that $Q$ satisfies all the assumptions in \ref{Bbbk} and the ideal $\mathfrak{r} = (\varepsilon)$ generated by $\varepsilon$ is a pseudo-radical in $Q$ with $\mathfrak{r}^2=0$. Hence we are in the setting of \stpref{nilpotent} so the results and definitions from \ref{summary-of-HJ2021} apply. The category $Q$ has only one object, which we denote by ``$*$''. In this case the funtors $\Fq{*}$ and $\Sq{*}$ from \ref{recall-F-G} and \ref{recall-C-K} are given by
\begin{equation*}
  \Fq{*}(M) \,=\,
  \big(
  M^2\,,
\left[\begin{smallmatrix} 
       0 & 0 \\ 1 & 0
    \end{smallmatrix}\right] \big)
  \qquad \text{and} \qquad
  \Sq{*}(M) \,=\, (M,0)  
\end{equation*}
for $M \in \lMod{\alg}$; elements in $M^2$ are viewed as column vectors. The stalk differential module $\stalkco{*} = \Sq{*}(\Bbbk) = (\Bbbk,0)$ in $\operatorname{Diff}(\Bbbk)$, see \ref{consequences-of-srp}, has the following projective resolution:
\begin{equation*}
  \xymatrix{
    \cdots 
    \ar[r]^-{\left[\begin{smallmatrix} 
       0 & 0 \\ 1 & 0
    \end{smallmatrix}\right]}     
    &
    \Fq{*}(\Bbbk) 
    \ar[r]^-{\left[\begin{smallmatrix} 
       0 & 0 \\ 1 & 0
    \end{smallmatrix}\right]} 
    &
    \Fq{*}(\Bbbk) 
    \ar[r]^-{\left[\begin{smallmatrix} 
       0 & 0 \\ 1 & 0
    \end{smallmatrix}\right]} 
    &
    \Fq{*}(\Bbbk) 
    \ar[r]^-{\left[\begin{smallmatrix} 
       1 & 0
    \end{smallmatrix}\right]} 
    &
    \stalkco{*} \ar[r] & 0
  }.
\end{equation*}
Let $(X,d)$ be in $\operatorname{Diff}(\alg)$. Applying the functor $\Hom{\mspace{1mu}\operatorname{Diff}(\Bbbk)}{-}{(X,d)}$ to the non-augmented version of the projective resolution above, and using the natural isomorphism 
\begin{equation*}
  \Hom{\mspace{1mu}\operatorname{Diff}(\Bbbk)}{\Fq{*}(\Bbbk)}{(X,d)} \,\cong\,  \Hom{\Bbbk}{\Bbbk}{X} \,\cong\, X \vspace*{0.5ex}
\end{equation*}
from \ref{recall-F-G}, one gets the complex \smash{$0 \longrightarrow X \stackrel{d}{\longrightarrow} X \stackrel{d}{\longrightarrow} X \stackrel{d}{\longrightarrow} \cdots$}. It follows that the cohomology $\cH[i]{*}(X,d)$ for $i>0$, see \ref{consequences-of-srp}, is the ordinary (co)homology $\operatorname{H}(X,d) = \Ker{d}/\Im{d}$ for differential modules. The same is true for homology, i.e.~$\hH[i]{*} = \operatorname{H}$ for every $i>0$. Thus:
\begin{itemlist}
\item A differential $\alg$-module $(X,d)$ is exact in the sense of \cite[Dfn.~4.1]{HJ2021} (see also \ref{summary-of-HJ2021}) if and only if $\operatorname{H}(X,d) = 0$, that is, $\Ker{d}=\Im{d}$. This follows from \cite[Thm.~7.1]{HJ2021}.

\item A morphism $\varphi$ of differential $\alg$-modules is a weak equivalence in the projective (or injective) model structure on $\operatorname{Diff}(\alg)$, see \cite[Thm.~6.1 and Prop.~6.3]{HJ2021} (and \ref{summary-of-HJ2021}) if and only if $\operatorname{H}(\varphi)$ is an isomorphism. This follows from \cite[Thm.~7.2]{HJ2021}.
\end{itemlist}
\end{ipg}

As $x^2=0$ holds in $\alg$ the pair $(\alg,x)$ is a differential $\alg$-module. The goal of this appendix is to prove the next result; it contains Theorem~A from the Introduction and shows that 
the conclusions in parts \prtlbl{a}--\prtlbl{c} of Theorem~B (from the Introduction) may fail without further assumptions on $Q$ or $\alg$. Note that the category $Q$ under investigation in this appendix has a loop, and the ring $\alg$ has infinite global dimension, so neither assumption \rqmlbl{1} nor \rqmlbl{2} in Theorem~B is satisfied in this case.

\begin{thm}
  \label{thm:counterexample}
  Let the rings $\Bbbk$ and $\alg$ be as above. The following assertions hold.
  \begin{prt}
  \item $(\alg,x)$ is a strictly perfect object in $\operatorname{Diff}(\alg)$.
  \item $(\alg,x)$ does not belong to ${}^\perp\lE{}$.
  \item $(\alg,x)$ is neither perfect nor compact in $\operatorname{Ho}(\operatorname{Diff}(\alg))$.
  \end{prt}
\end{thm}

As stated in part \prtlbl{b} of the theorem above, $(\alg,x)$ does not belong to ${}^\perp\lE{}$; that is, $(\alg,x)$ is not a cofibrant object in the projective model structure on $\operatorname{Diff}(\alg)$. In \prpref{cof-repl} below we give a cofibrant replacement of $(\alg,x)$; to construct it we need Zeckendorf expansions, which we now explain.

\begin{ipg}
  \label{Zeckendorf}
  Consider the Fibonacci numbers $F_{-2}, F_{-1}, F_0, F_1, F_2, F_3,\ldots$ defined by
  \begin{equation*}
     F_{-2}=0 \ \ , \ \ F_{-1}=1
     \qquad \text{and} \qquad
     F_i \,=\, F_{i-1}+F_{i-2} 
     \ \ \text{ for } \ \ i \geqslant 0\;;
  \end{equation*}
  that is $F_0=1$,\, $F_1=2$,\, $F_2=3$,\, $F_3=5$,\, $F_4=8$,\, $F_5=13$,\, $F_6=21$,\, $F_7=34$ etc. 
  
  Zeckendorf's theorem \cite{Zeckendorf} (see also Dekking \cite{Dekking}) asserts that for every $n \in \mathbb{N}_0$ there exists a unique sequence $[n]_\mathrm{Z} = \cdots d_2d_1d_0$ of digits $d_i \in \{0,1\}$, called the \emph{Zeckendorf expansion} of $n$, satisfying
\begin{equation*}
  n \,=\, d_0F_0 + d_1F_1 + d_2F_2 + \cdots 
  \qquad \text{and} \qquad 
  d_i\cdot d_{i+1}=0
  \quad \text{for all} \quad i \geqslant 0\;.
\end{equation*}  
The last condition means that in the Zeckendorf expansion the digit combination $\cdots 11 \cdots$ (with two consecutive 1's) is not allowed. Without this condition the Zeckendorf expansion would not be unique as e.g.~$6 = F_0+F_1+F_2 = F_0+F_3$. 

A number $n \in \mathbb{N}_0$ is said to have \emph{even} respectively, \emph{odd}, Zeckendorf expansion if one has $d_0=0$, respectively, $d_0=1$, in the Zeckendorf expansion $[n]_\mathrm{Z} = \cdots d_2d_1d_0$ of $n$. For example, the numbers $49$ and $50$ have even Zeckendorf expansions $[49]_\mathrm{Z} = 1010001\underline{0}$ and $[50]_\mathrm{Z} = 1010010\underline{0}$ while 51 has odd Zeckendorf expansion $[51]_\mathrm{Z} = 1010010\underline{1}$.

For every $i \geqslant 1$ we let $e(i)$, respectively, $o(i)$ be the $i^\mathrm{th}$ number in $\mathbb{N}_0$ whose Zeckendorf expansion is even, respectively, odd. The first few values of the functions $e$ and $o$ are:   \vspace*{0.5ex}
\begin{equation*}
  \begin{array}{|c|cccccccccccccccc|}
  \hline
  \pmb{i} & \pmb{1} & \pmb{2} & \pmb{3} & \pmb{4} & \pmb{5} & \pmb{6} & \pmb{7} & \pmb{8} & \pmb{9} & \pmb{10} & \pmb{11} & \pmb{12} & \pmb{13} & \pmb{14} & \pmb{15} & \cdots
  \\ \hline \hline
  e(i) & 0 & 2 & 3 & 5 & 7 & 8 & 10 & 11 & 13 & 15 & 16 & 18 & 20 & 21 & 23 & \cdots
  \\ \hline
  o(i) & 1 & 4 & 6 & 9 & 12 & 14 & 17 & 19 & 22 & 25 & 27 & 30 & 33 & 35 & 38 & \cdots 
  \\ \hline
  \end{array}
  \vspace*{0.5ex}
\end{equation*}
The number sequences $e(1), e(2), e(3),\ldots$ and $o(1), o(2), o(3),\ldots$ can be found in the OEIS \cite[\href{https://oeis.org/A022342}{A022342} and \href{https://oeis.org/A003622}{A003622}]{OEIS}. By  Zeckendorf's theorem, mentioned above, one has 
\begin{equation*}
  \{e(1), e(2), e(3),\ldots\} \ \uplus \ \{o(1), o(2), o(3),\ldots\} \,=\, \mathbb{N}_0\;.
\end{equation*}
We shall need the following well-known formula, also recorded in the OEIS \cite[\href{https://oeis.org/A003622}{A003622}]{OEIS},
\begin{equation}
  \label{eq:o-e}
  o(i) \,=\,  e(e(i)+1)+1
  \quad \text{for every} \quad i \geqslant 1\;.
\end{equation}
\end{ipg} 

\setcounter{MaxMatrixCols}{25}
\newcommand{\z}{\color{lightgray}0}

\begin{dfn}
  \label{dfn:partial}
  Define an $\mathbb{N} \times \mathbb{N}$ matrix $\partial$ with entries from $\alg$ as follows:
  \begin{itemlist}
  \item If $i \in \mathbb{N}$ is a number with odd Zeckendorf expansion ($i= 1, 4, 6, 9, \ldots$), then 
  \begin{equation}
    \label{eq:i-odd}
    \begin{gathered}
    \partial_{ij} \,=\, 
    \left\{\!\!\!
      \begin{array}{ccl}
        x & \text{if} & j=e(i+1) \\
        0 & \multicolumn{2}{l}{\text{otherwise}\,.}
      \end{array}
    \right.
    \end{gathered}
  \end{equation}  

  \item  If $i \in \mathbb{N}$ is a number with even Zeckendorf expansion ($i= 2, 3, 5, 7, 8, \ldots$), then
  \begin{equation}
    \label{eq:i-even}
    \begin{gathered}
    \partial_{ij} \,=\, 
    \left\{\!\!\!
      \begin{array}{ccl}
        x & \text{if} & j=e(i+1) \\
        y & \text{if} & j=e(i+1)+1 \\
        0 & \multicolumn{2}{l}{\text{otherwise}\,.} 
      \end{array}
    \right.
    \end{gathered}
  \end{equation}  

  \end{itemlist}
  Thus, the upper left $13 \times 20$ corner of $\partial$ looks like this: \vspace*{0.5ex}
  \begin{equation*}
    \partial \,=\,
    \begin{bmatrix}
      \z & x & \z & \z & \z & \z & \z & \z & \z & \z & \z & \z & \z & \z & \z & \z & \z & \z & \z & \z & \cdots
      \\ 
      \z & \z & x & y & \z & \z & \z & \z & \z & \z & \z & \z & \z & \z & \z & \z & \z & \z & \z & \z & \cdots
      \\ 
      \z & \z & \z & \z & x & y & \z & \z & \z & \z & \z & \z & \z & \z & \z & \z & \z & \z & \z & \z & \cdots
      \\ 
      \z & \z & \z & \z & \z & \z & x & \z & \z & \z & \z & \z & \z & \z & \z & \z & \z & \z & \z & \z & \cdots
      \\ 
      \z & \z & \z & \z & \z & \z & \z & x & y & \z & \z & \z & \z & \z & \z & \z & \z & \z & \z & \z & \cdots
      \\ 
      \z & \z & \z & \z & \z & \z & \z & \z & \z & x & \z & \z & \z & \z & \z & \z & \z & \z & \z & \z & \cdots
      \\ 
      \z & \z & \z & \z & \z & \z & \z & \z & \z & \z & x & y & \z & \z & \z & \z & \z & \z & \z & \z & \cdots
      \\ 
      \z & \z & \z & \z & \z & \z & \z & \z & \z & \z & \z & \z & x & y & \z & \z & \z & \z & \z & \z & \cdots
      \\ 
      \z & \z & \z & \z & \z & \z & \z & \z & \z & \z & \z & \z & \z & \z & x & \z & \z & \z & \z & \z & \cdots
      \\ 
      \z & \z & \z & \z & \z & \z & \z & \z & \z & \z & \z & \z & \z & \z & \z & x & y & \z & \z & \z & \cdots
      \\ 
      \z & \z & \z & \z & \z & \z & \z & \z & \z & \z & \z & \z & \z & \z & \z & \z & \z & x & y & \z & \cdots
      \\ 
      \z & \z & \z & \z & \z & \z & \z & \z & \z & \z & \z & \z & \z & \z & \z & \z & \z & \z & \z & x & \cdots
      \\ 
      \z & \z & \z & \z & \z & \z & \z & \z & \z & \z & \z & \z & \z & \z & \z & \z & \z & \z & \z & \z & \cdots
      \\ 
      \vdots & \vdots & \vdots & \vdots & \vdots & \vdots & \vdots & \vdots & \vdots & \vdots & \vdots & \vdots & \vdots & \vdots & \vdots & \vdots & \vdots & \vdots & \vdots & \vdots & \ddots
    \end{bmatrix}
    \vspace*{0.5ex}
  \end{equation*}
  For every $n \geqslant 1$ let $\partial_n$ denote the $n\times n$ upper left  corner of $\partial$, for example
  \begin{equation*}
    \partial_1 \,=\,
    \begin{bmatrix}
      0
    \end{bmatrix}\;, \quad
    \partial_2\,=\,
    \begin{bmatrix}
      0 & x \\
      0 & 0
    \end{bmatrix}\;, \quad
    \partial_3\,=\,
    \begin{bmatrix}
      0 & x & 0 \\
      0 & 0 & x \\
      0 & 0 & 0
    \end{bmatrix}\;, \quad \text{and} \quad
    \partial_4\,=\,
    \begin{bmatrix}
      0 & x & 0 & 0 \\
      0 & 0 & x & y \\
      0 & 0 & 0 & 0 \\
      0 & 0 & 0 & 0
    \end{bmatrix}\;.
  \end{equation*}
\end{dfn}

\begin{rmk}
   Note that $\partial$ is both row and column finite, that is, each row and each column in $\partial$ contains only finitely many (actually at most two) non-zero entries. Futhermore, every entry in $\partial$ is either $0$, $x$ or $y$.
\end{rmk}

We view elements in $\alg^{(\mathbb{N})}$ and $\alg^n$ as column vectors.

\begin{lem}
  \label{lem:dsquared}
  One has $\partial^2=0$ and $\partial_n^2=0$ for every $n \geqslant 1$; whence the pairs $(\alg^{(\mathbb{N})},\partial)$ and $(\alg^n,\partial_n)$ are differential $\alg$-modules.
\end{lem}

\begin{proof}
  Note that for each $n \geqslant 1$ we can write $\partial$ as a block matrix of the form
  \begin{equation*}
    \partial \,=\,
    \left[
    \begin{array}{c|c}
      \partial_n & * \\
      \hline
      0 & *
    \end{array}
    \right]
    \qquad \text{and hence $\partial^2$ has the form} \qquad
    \partial^2 \,=\,
    \left[
    \begin{array}{c|c}
      \partial^2_n & * \\
      \hline
      0 & *
    \end{array}
    \right]\;.    
  \end{equation*}
  Thus, it suffices to argue that $\partial^2=0$ holds. Fix $i,j \in \mathbb{N}$. Entry $(i,j)$ in the matrix $\partial^2$ is the sum \smash{$s_{i\!j} = \sum_{k \in \mathbb{N}} \partial_{ik}\partial_{kj}$}, so we must argue that $s_{i\!j}$ is zero.
  
  If $i$ has odd Zeckendorf expansion, then by \eqref{i-odd} one has \smash{$s_{i\!j} = x\cdot\partial_{e(i+1),j}$}\,, which is zero as \smash{$\partial_{e(i+1),j}$} belongs to $\{0,x,y\}$ and $x^2=xy=0$ holds in $\alg$. 
  
  If $i$ has even Zeckendorf expansion, then \eqref{i-even} yields \smash{$s_{i\!j} = x\cdot \partial_{e(i+1),j} + y\cdot \partial_{e(i+1)+1,j}$}\,. As above, the first term in this sum is zero. As the number $i$ has even Zeckendorf expansion, it has the form $i=e(\ell)$ for some $\ell >1$. Thus $e(i+1)+1 = e(e(\ell)+1)+1 = o(\ell)$ by \eqref{o-e}, so the number $e(i+1)+1$ has odd Zeckendorf expansion. Hence \smash{$\partial_{e(i+1)+1,j}$} is in $\{0,x\}$ by \eqref{i-odd}, and since $yx=0$ holds in $\alg$ the second term in the sum $s_{i\!j}$ is zero as well.
\end{proof}

\begin{lem}
  \label{lem:ses}
  For every $n \geqslant 1$ there is a short exact sequence of differential $\alg$-modules,
\begin{equation*}
  \xymatrix@C=1.3pc{
  0 \ar[r] &
  (\alg^n,\partial_n) 
  \ar[rr]^-{\iota_n \;=\; \left[\begin{smallmatrix} 1 \\ 0\end{smallmatrix}\right]} & &
  (\alg^{n+1},\partial_{n+1}) \ar[rr]^-{\left[\begin{smallmatrix} 0 & 1 \end{smallmatrix}\right]} & &
  (\alg,0) \ar[r] & 0
  }
\end{equation*}  
\end{lem}

\begin{proof}
Since the differential module $(\alg^{n+1},\partial_{n+1})$ can be written as
\begin{equation*}
  (\alg^{n+1},\partial_{n+1})
  \,=\, 
  \bigg(
  \mspace{-10mu}
  {
  \renewcommand{\arraystretch}{0.8}  
  \begin{array}{c}
    \mspace{5mu}\alg^n 
    \\ \oplus \\ 
    \alg 
  \end{array}
  }\mspace{-7mu},
  \left[
  \begin{array}{c|c}
    \partial_n & * \\
    \hline
    0 & 0
  \end{array}
  \right]
  \bigg)\;,
\end{equation*}
the assertion follows immediately.
\end{proof}

\begin{lem}
  \label{lem:colim}
  In the abelian category $\operatorname{Diff}(\alg)$, the (filtered) colimt of the telescope
  \begin{equation}
    \label{eq:filtration}
    \xymatrix{
      (A,\partial_1)\, \ar@{>->}[r]^-{\iota_1} &
      (A^2,\partial_2)\, \ar@{>->}[r]^-{\iota_2} &
      (A^3,\partial_3)\, \ar@{>->}[r]^-{\iota_3} &
      \cdots
    }
  \end{equation}
  is the differential module $(\alg^{(\mathbb{N})},\partial)$. In symbols: \smash{$(\alg^{(\mathbb{N})},\partial) = \varinjlim_{\,n \geqslant 1}(\alg^n,\partial_n)$}.
\end{lem}

\begin{proof}
  This is evident from the definitions of $(\alg^n,\partial_n)$ and $(\alg^{(\mathbb{N})},\partial)$.
\end{proof}

\begin{prp}
  \label{prp:perp-E}
  The differential modules $(\alg^{(\mathbb{N})},\partial)$ and $(\alg^n,\partial_n)$ for $n \geqslant 1$ belong to ${}^\perp\lE{}$.
\end{prp}

\begin{proof}
  The class $\mathcal{P}$ in \prpref{E-perp-inclusions}\prtlbl{a} is $\mathcal{P} = \{ (P,0) \, | \, P \in \lPrj{\alg} \}$. By \lemref{ses} each $\iota_n$ is a monomorphism with cokernel $(A,0)$, which belongs to $\mathcal{P}$. Recall from \dfnref{partial} that $\partial_1=0$ holds, and therefore $(A,\partial_1)$ is also equal to $(A,0)$. Hence \eqref{filtration} is a $\mathcal{P}$-filtration of \smash{$(\alg^{(\mathbb{N})},\partial) = \varinjlim_{\,n \geqslant 1}(\alg^n,\partial_n)$} and the assertion follows from \prpref{E-perp-inclusions}\prtlbl{a}.
\end{proof}

To parse the next result, recall from Neeman \cite[Dfn.~1.6.4]{Nee} the definition of \emph{homotopy colimits} in a triangulated category. Recall from \ref{summary-of-HJ2021} that ${}^\perp\lE{}$ is a Frobenius category and its stable category is therefore triangulated by Happel \cite[Chap.~I\S2]{Happel}.

\begin{cor}
  \label{cor:hocolim}
  In the stable category of the Frobenius category ${}^\perp\lE{}$, the object $(\alg^{(\mathbb{N})},\partial)$ is the homotopy colimit of the sequence \eqref{filtration}. In symbols: \smash{$(\alg^{(\mathbb{N})},\partial) = \mathrm{hocolim}_{\,n \geqslant 1}(\alg^n,\partial_n)$}.
\end{cor}

\begin{proof}
  We know from \prpref{perp-E} that $(\alg^n,\partial_n)$ and $(\alg^{(\mathbb{N})},\partial)$ are objects in ${}^\perp\lE{}$. As \smash{$(\alg^{(\mathbb{N})},\partial)$} is the direct limit of \eqref{filtration} in $\operatorname{Diff}(\alg)$ there is a short exact sequence in $\operatorname{Diff}(\alg)$,
\begin{equation}
  \label{eq:lim-ses}
  \xymatrix@C=1.7pc{
    0 \ar[r] 
    & 
    \bigoplus_{n \in \mathbb{N}} (\alg^n,\partial_n)
    \ar[r]^-{\iota}
    & 
    \bigoplus_{n \in \mathbb{N}} (\alg^n,\partial_n)
    \ar[r]
    &    
    (\alg^{(\mathbb{N})},\partial)
    \ar[r]
    & 
    0
  },
\end{equation}  
where $\sigma$ is the morphism given by
\begin{equation*}
  \iota(x_1,x_2,x_3,\ldots) \,=\, (x_1,\,x_2-\iota_1(x_1),\, x_3-\iota_2(x_2),\,\ldots)\;.
\end{equation*}
As the class ${}^\perp\lE{}$ is closed under coproducts in $\operatorname{Diff}(\alg)$, the sequence \eqref{lim-ses} is a conflation in the Frobenius category ${}^\perp\lE{}$. It now follows from Happel \cite[Chap.~I\S2.7]{Happel} that there is a distinguished triangle,
\begin{equation*}
  \xymatrix@C=1.7pc{
    \bigoplus_{n \in \mathbb{N}} (\alg^n,\partial_n)
    \ar[r]^-{\iota}
    & 
    \bigoplus_{n \in \mathbb{N}} (\alg^n,\partial_n)
    \ar[r]
    &    
    (\alg^{(\mathbb{N})},\partial)
    \ar[r]
    & 
    {}
  },
\end{equation*}  
in the stable category of ${}^\perp\lE{}$. By definition, see \cite[Dfn.~1.6.4]{Nee}, this means that $(\alg^{(\mathbb{N})},\partial)$ is the homotopy colimit of the sequence \eqref{filtration} the stable category of ${}^\perp\lE{}$.
\end{proof}

\begin{lem}
  \label{lem:relation}
  For $a,b \in A$ one has $ax+by=0$ if and only if $a \in (x,y)$ and $b \in (x)$.
\end{lem}

\begin{proof}
  This follows easily from the fact that $\alg$ is free over $\Bbbk$ with basis $\{1,x,y,y^2,y^3,\ldots\}$ and from the relations $x^2=xy=0$ in $\alg$.
\end{proof}

\begin{prp}
  \label{prp:cof-repl}
  There is a weak equivalence in $\operatorname{Diff}(\alg)$,
  \begin{equation*}
     \varphi \colon (\alg^{(\mathbb{N})},\partial) \stackrel{\sim}{\longrightarrow} (\alg,x)
     \qquad \text{given by} \qquad
     (a_1,a_2,a_3,\ldots) \longmapsto y\,a_1\;;
  \end{equation*}  
  so $(\alg^{(\mathbb{N})},\partial)$ is a cofibrant replacement of $(\alg,x)$ in the projective model structure on $\operatorname{Diff}(\alg)$.
\end{prp}

\begin{proof}
  For every element $(a_1,a_2,\ldots)$ in $\alg^{(\mathbb{N})}$ one has 
  \begin{equation*}
    \varphi\partial(a_1,a_2,\ldots) \,=\,
    \varphi(xa_2,xa_3+ya_4,\ldots) \,=\, yxa_2 \,=\, 0 \,=\, xya_1 \,=\, x\varphi(a_1,a_2,\ldots)\;,
  \end{equation*}
  so $\varphi$ is a morphism of differential $\alg$-modules. As $(\alg^{(\mathbb{N})},\partial)$ is in ${}^\perp\lE{}$ by \prpref{perp-E}, it is a cofibrant object in the projective model structure on $\operatorname{Diff}(\alg)$, see \ref{summary-of-HJ2021}. Below we show
\begin{equation}
  \label{eq:H}
  \operatorname{H}(\alg^{(\mathbb{N})},\partial)
  \,=\, \alg/(x) \mspace{1mu}\oplus\mspace{1mu} 0 \mspace{1mu}\oplus\mspace{1mu} 0 \mspace{1mu}\oplus\mspace{1mu} \cdots\;,
\end{equation}  
and clearly $\operatorname{H}(\alg,x) = (x,y)/(x)$. Thus, the map $\operatorname{H}(\varphi)$ is the map $\alg/(x) \to (x,y)/(x)$ given by multiplication with the element $y$. Since this is evidently an isomorphism, $\varphi$ is a weak equivalence by \ref{Diff}. It remains to prove \eqref{H}, which requires a bit more work:
  
It follows directly from the definition of $\partial$, see \dfnref{partial}, that its image is a direct sum, $\Im{\partial} = \bigoplus_{i \in \mathbb{N}} I_i\mspace{2mu}$, where 
\begin{equation}
  \label{eq:Im}
  I_i \,=\, 
  \left\{\!\!\!
    \begin{array}{ccl}
      (x) & \text{if} & \text{$i$\, has odd Zeckendorf expansion}
      \\
      (x,y) & \text{if} & \text{$i$\, has even Zeckendorf expansion}
    \end{array}
  \right.
  \quad \ (i \in \mathbb{N})\;.
\end{equation}
That is, 
\begin{equation*}
  \Im{\partial} \ = \ I_1 \oplus I_2 \oplus I_3 \oplus I_4 \oplus I_5 \oplus \cdots \ = \  (x) \oplus (x,y) \oplus (x,y) \oplus (x) \oplus (x,y) \oplus \cdots\;.
\end{equation*}
If $(a_1,a_2,a_3,\ldots) \in \alg^{(\mathbb{N})}$ belongs to $\Ker{\partial}$, then the definition of $\partial$ shows that one has
\begin{equation*} 
  \left\{\!\!\!
    \begin{array}{rcl}
      xa_{e(i+1)}\,=\,0 & \text{if} & \text{$i$\, has odd Zeckendorf expansion}
      \\
      xa_{e(i+1)}+ya_{e(i+1)+1}\,=\,0 & \text{if} & \text{$i$\, has even Zeckendorf expansion}
    \end{array}
  \right.
  \quad \ (i \in \mathbb{N})\;.  
\end{equation*}
By \lemref{relation} this means that 
\begin{equation*} 
  \left\{\!\!\!
    \begin{array}{ll}
      a_{e(i+1)+1} \in (x) & \text{for every $i \in \mathbb{N}$ with even Zeckendorf expansion, and}
      \\
      \mspace{17mu} a_{e(i+1)} \in (x,y) & \text{for every $i \in \mathbb{N}$\;.}
    \end{array}
  \right.
\end{equation*}
When $i$ ranges over $\mathbb{N}$, the values $e(i+1)$ ranges over all natural numbers with even Zeckendorf expansion; thus the last condition above is equivalent to saying that $a_{\!j} \in (x,y)$ for every $j \in \mathbb{N}$ with even Zeckendorf expansion. As just mentioned, every $i \in \mathbb{N}$ with even Zeckendorf expansion has the form $i = e(\ell)$ for some $\ell \geqslant 2$, so the first condition above can be phrased as: $a_{e(e(\ell)+1)+1} \in (x)$ for every $\ell \geqslant 2$. By using \eqref{o-e}, this is exactly the same as: $a_{o(\ell)} \in (x)$ for every $\ell \geqslant 2$. When $\ell$ ranges over numbers $\geqslant 2$, the values $o(\ell)$ ranges over all natural numbers $j$ with odd Zeckendorf expansion -- except for $j=1$. Thus the first condition above is equivalent to saying that $a_{\!j} \in (x)$ for every $j > 1$ with odd Zeckendorf expansion. In conclusion: an element \smash{$(a_1,a_2,a_3,\ldots) \in \alg^{(\mathbb{N})}$} belongs to $\Ker{\partial}$ if and only if
\begin{equation*} 
  \left\{\!\!\!
    \begin{array}{ll}
      a_{\!j} \in (x) & \text{for every $j > 1$ with odd Zeckendorf expansion, and}
      \\
      a_{\!j} \in (x,y) & \text{for every $j \in \mathbb{N}$ with even Zeckendorf expansion\;;}
    \end{array}
  \right.
\end{equation*}
note that there is no condition on the element $a_1 \in \alg$. Consequently, the kernel of $\partial$ is a direct sum, $\Ker{\partial} = \bigoplus_{j \in \mathbb{N}} K_j\mspace{2mu}$, where 
\begin{equation}
  \label{eq:Ker}
  K_j \,=\, 
  \left\{\!\!\!
    \begin{array}{ccl}
      \alg & \text{if} & \text{$j=1$}
      \\
      (x) & \text{if} & \text{$j>1$\, has odd Zeckendorf expansion}
      \\
      (x,y) & \text{if} & \text{$j$\, has even Zeckendorf expansion}
    \end{array}
  \right.
  \quad \ (j \in \mathbb{N})\;.
\end{equation}
That is, 
\begin{equation*}
  \Ker{\partial} \ = \ K_1 \oplus K_2 \oplus K_3 \oplus K_4 \oplus K_5 \oplus \cdots \ = \  \alg \oplus (x,y) \oplus (x,y) \oplus (x) \oplus (x,y) \oplus \cdots\;.
\end{equation*}
Comparing \eqref{Im} and \eqref{Ker} we see that $I_1 = (x) \subset \alg = K_1$ while $I_i = K_i$ for every $i>1$. As one has $\Ker{\partial}/\Im{\partial} = \bigoplus_{i \in \mathbb{N}} K_i/I_i$, the equality \eqref{H} follows.
\end{proof}

\begin{rmk}
  \label{rmk:varphi-row}
  Since we view elements in $\alg^{(\mathbb{N})}$ as column vectors, the map $\varphi$ from \prpref{cof-repl} could be written as the $1 \times \mathbb{N}$ matrix $\varphi = [\mspace{2mu}y \ \ 0 \ \ 0 \ \ \cdots]$.
\end{rmk}

\begin{proof}[Proof of \thmref{counterexample}]
  \proofoftag{a} It is immediate from \dfnref{perfect} that $(\alg,x)$ is a strictly perfect object in $\operatorname{Diff}(\alg)$. (In the situation at hand, the category $Q$ has just one object, so every object in $\lMod{Q,\,\alg} \simeq \operatorname{Diff}(\alg)$ actually has finite support.)
  
  \proofoftag{b} Note that there is an epimorphism of differential $\alg$-modules:
\begin{equation*}
  \xymatrix@C=3pc{
  {
  \Fq{*}(\alg) \,=\,
  \big(
  \alg^2\,,
    \left[\begin{smallmatrix} 
       0 & 0 \\ 1 & 0
    \end{smallmatrix}\right] 
  \big)
  }
  \ar@{->>}[r]^-{
    [\begin{smallmatrix} 
       1 & x
    \end{smallmatrix}] 
  } &
  (\alg,x)
  }.
\end{equation*}
By adding this epimorphism to the morphism $\varphi$ from \prpref{cof-repl} we get, in view of \rmkref{varphi-row}, the epimorphism
\begin{equation*}
  \xymatrix@C=7pc{
  {(L,d) \,:=\,
  \Bigg(
  \mspace{-3mu}
  {
  \renewcommand{\arraystretch}{0.85}  
  \begin{array}{c}
    \mspace{5mu}\alg^2
    \\ \oplus \\ 
    \alg^{(\mathbb{N})} 
  \end{array}
  }\mspace{-7mu},
  \left[
  \begin{array}{c|c}
    {\!\!\begin{smallmatrix} 
       0 & 0 \\ 1 & 0
    \end{smallmatrix}\!\!} & 0 \\
    \hline
    0 & \partial
  \end{array}
  \right]
  \Bigg)
  }
  \ar@{->>}[r]^-{\psi \ = \
    [\begin{smallmatrix} 
       1 & x & \!|\! & y & 0 & 0 & \cdots
    \end{smallmatrix}] 
  } &
  (\alg,x)
  }.
\end{equation*}
As $\Fq{*}(\alg)$ is exact and $\varphi$ is a weak equivalence, the morphism $\psi$ is a weak equivalence too. Now, consider in $\operatorname{Diff}(\alg)$ the short exact sequence,
\begin{equation*}
  \xymatrix@C=1.5pc{
    0 \ar[r] & \Ker{\psi} \ar[r] & (L,d) \ar[r]^-{\psi} & (\alg,x) \ar[r] & 0
  }.
\end{equation*}
In the projective model structure on $\operatorname{Diff}(\alg)$ every object is fibrant, see \ref{summary-of-HJ2021}, and hence the fibrations in this model structure are nothing but epimorphisms, see \cite[Dfn.~5.1]{Hovey02}. Thus, $\psi$ is a fibration. Since $\psi$ is also a weak equivalence, it follows from \cite[Lem.~5.8]{Hovey02} that its kernel is exact, that is, $\Ker{\psi} \in \lE{}$. (Using the same method as in the proof of \prpref{cof-repl}, it is not hard to show that $\operatorname{H}(\Ker{\psi})=0$. Some readers may prefer this more direct approach.) Thus, \emph{if} $(\alg,x)$ were an object in ${}^\perp\lE{}$, then $\Ext{\mspace{1mu}\operatorname{Diff}(\alg)}{1}{(\alg,x)}{\Ker{\psi}}=0$ would hold and hence the short exact sequence above would split. But this is impossible; indeed, we now prove that $\psi$ does not have a right inverse (= a section).

Suppose that $\sigma \colon (\alg,x) \to (L,d)$ is a right inverse of $\psi$ in $\operatorname{Diff}(\alg)$. Every $\alg$-linear map $\alg \to \alg^2 \oplus \alg^{(\mathbb{N})}$; in particular, $\sigma$, has the form
\begin{equation*}
  \sigma(a) \,=\, (b'a\,,\, b''a\,,\, b_1a\,,\, b_2a\,,\, \ldots)
  \qquad , \qquad a \in \alg
\end{equation*}
for some fixed elements $b',b'',b_1,b_2,\ldots$ in $\alg$ (of which only finitely many are non-zero). As $\sigma$ is a morphism of differential modules, one has $d\sigma = \sigma x\mspace{1mu}$; in particular:
\begin{equation*}
  (0\,,\, b',\, b_2x\,,\, \ldots) \,=\,
  d(b',\, b'',\, b_1\,,\, \ldots) \,=\, d\sigma(1) \,=\, \sigma(x) \,=\, (b'x\,,\, b''x\,,\, b_1x\,,\, \ldots)\;.
\end{equation*}
By comparing the second coordinates (which is all we need for the argument), we see that $b' = b''x$ holds. As $\sigma$ is a right inverse of $\psi$ we have:
\begin{equation*}
  1 \,=\, \psi\sigma(1) \,=\, \psi(b',\, b'',\, b_1\,,\, b_2\,,\, \ldots) \,=\, b' + b''x + b_1y \,=\, 2b''x+b_1y\;,
\end{equation*}
where the last equality holds by the just established relation $b' = b''x$. But this is impossible since $1$ does not belong to the maximal ideal $(x,y) \subset \alg$.

\proofoftag{c} To prove that $(\alg,x)$ is neither perfect nor compact in $\QSD{Q}{\alg} \simeq \operatorname{Ho}(\operatorname{Diff}(\alg))$, it suffices by \corref{perfect-is-compact-in-D} to argue that $(\alg,x)$ is not compact. To this end, let $\mathcal{T}$ be the stable category of the Frobenius category ${}^\perp\lE{}$ and recall from \ref{amc} that the inclusion functor ${}^\perp\lE{} \to \operatorname{Diff}(\alg)$ induces a triangulated equivalence $\mathcal{T} \to \operatorname{Ho}(\operatorname{Diff}(\alg))$. By \prpref{cof-repl} the objects
$(\alg,x)$ and $(\alg^{(\mathbb{N})},\partial)$ are isomorphic in $\operatorname{Ho}(\operatorname{Diff}(\alg))$ and by \prpref{perp-E} the object $(\alg^{(\mathbb{N})},\partial)$ is in $\mathcal{T}$. Hence $(\alg,x)$ will be a compact object in $\operatorname{Ho}(\operatorname{Diff}(\alg))$ if and only if $(\alg^{(\mathbb{N})},\partial)$ is a compact object in $\mathcal{T}$. We now show that $(\alg^{(\mathbb{N})},\partial)$ is \emph{not} compact in $\mathcal{T}$.

By \corref{hocolim} we have \smash{$(\alg^{(\mathbb{N})},\partial) = \mathrm{hocolim}_{\,n \geqslant 1}(\alg^n,\partial_n)$} in $\mathcal{T}$, so \emph{if} $(\alg^{(\mathbb{N})},\partial)$ is compact, then the canonical homomorphism
\begin{equation}
  \label{eq:canonical-map}
  \textstyle
  \varinjlim_{\mspace{1mu}n \geqslant 1} \Hom{\mathcal{T}}{(\alg^{(\mathbb{N})},\partial)}{(\alg^n,\partial_n)} \longrightarrow
  \Hom{\mathcal{T}}{(\alg^{(\mathbb{N})},\partial)}{(\alg^{(\mathbb{N})},\partial)} 
\end{equation}
is an isomorphism; see Neeman \cite[Lem.~2.8]{MR1308405} or Rouquier \cite[Lem.~3.12]{MR2434186}. In particular, the identity morphism on $(\alg^{(\mathbb{N})},\partial)$ is in the image of the map \eqref{canonical-map}, which implies that:
\begin{equation}
  \label{eq:summand-1}
  \text{In $\mathcal{T}$ the object $(\alg^{(\mathbb{N})},\partial)$ is a direct summand of some $(\alg^n,\partial_n)$\;.}
\end{equation}
Now consider the ring homomorphism $\alg \twoheadrightarrow \alg/(x,y) = \Bbbk$ and the following composition of functors, where the first functor is the inclusion and the third functor is the canonical one:
\begin{equation*}
  \xymatrix@C=2.2pc{
    {}^\perp\lE{} \ar[r] & 
    \operatorname{Diff}(\alg) \ar[r]^-{\Bbbk\mspace{1mu}\otimes_\alg -} &
    \operatorname{Diff}(\Bbbk) \ar[r] &
    \operatorname{Ho}(\operatorname{Diff}(\Bbbk))
  }.
\end{equation*}
The functor $\Bbbk\otimes_\alg-$ preserves projective objects, and every projective object in $\operatorname{Diff}(\Bbbk)$ is exact and hence mapped to zero in $\operatorname{Ho}(\operatorname{Diff}(\Bbbk))$. It follows that the composite functor above induces a functor $\Bbbk\otimes_\alg- \colon \mathcal{T} \to \operatorname{Ho}(\operatorname{Diff}(\Bbbk))$. 
From \eqref{summand-1} we now conclude that 
\begin{equation}
  \label{eq:summand-2}
  \text{In $\operatorname{Ho}(\operatorname{Diff}(\Bbbk))$ the object $(\Bbbk^{(\mathbb{N})},0)$ is a direct summand of some $(\Bbbk^n,0)$\;.}
\end{equation}
Here we have used that $\Bbbk\otimes_\alg\partial=0$ and $\Bbbk\otimes_\alg\partial_n=0$ since the matrices $\partial$ and $\partial_n$ have entries in the ideal $(x,y) \subset \alg$, see \dfnref{partial}. Finally, we apply the homology functor $\operatorname{H} \colon \operatorname{Ho}(\operatorname{Diff}(\Bbbk)) \to \lMod{\Bbbk}$. From \eqref{summand-2} we now conclude that in the category $\lMod{\Bbbk}$ the module \smash{$\Bbbk^{(\mathbb{N})}$} is a direct summand of $\Bbbk^n$ for some $n \geqslant 1$. But this is  impossible.
\end{proof}

\def\cprime{$'$} \def\soft#1{\leavevmode\setbox0=\hbox{h}\dimen7=\ht0\advance
  \dimen7 by-1ex\relax\if t#1\relax\rlap{\raise.6\dimen7
  \hbox{\kern.3ex\char'47}}#1\relax\else\if T#1\relax
  \rlap{\raise.5\dimen7\hbox{\kern1.3ex\char'47}}#1\relax \else\if
  d#1\relax\rlap{\raise.5\dimen7\hbox{\kern.9ex \char'47}}#1\relax\else\if
  D#1\relax\rlap{\raise.5\dimen7 \hbox{\kern1.4ex\char'47}}#1\relax\else\if
  l#1\relax \rlap{\raise.5\dimen7\hbox{\kern.4ex\char'47}}#1\relax \else\if
  L#1\relax\rlap{\raise.5\dimen7\hbox{\kern.7ex
  \char'47}}#1\relax\else\message{accent \string\soft \space #1 not
  defined!}#1\relax\fi\fi\fi\fi\fi\fi} \def\cprime{$'$}
  \providecommand{\arxiv}[2][AC]{\mbox{\href{http://arxiv.org/abs/#2}{\tt
  arXiv:#2 [math.#1]}}}
  \providecommand{\oldarxiv}[2][AC]{\mbox{\href{http://arxiv.org/abs/math/#2}{\sf
  arXiv:math/#2
  [math.#1]}}}\providecommand{\MR}[1]{\mbox{\href{http://www.ams.org/mathscinet-getitem?mr=#1}{#1}}}
  \renewcommand{\MR}[1]{\mbox{\href{http://www.ams.org/mathscinet-getitem?mr=#1}{#1}}}
\providecommand{\bysame}{\leavevmode\hbox to3em{\hrulefill}\thinspace}
\providecommand{\MR}{\relax\ifhmode\unskip\space\fi MR }
\providecommand{\MRhref}[2]{%
  \href{http://www.ams.org/mathscinet-getitem?mr=#1}{#2}
}
\providecommand{\href}[2]{#2}


\end{document}